\documentclass[a4paper,12pt]{article}

\usepackage[english]{babel}

\usepackage[utf8]{inputenc}
\usepackage{amsmath}
\usepackage{amsthm}
\usepackage{amssymb}
\usepackage{mathtools}
\usepackage{mathrsfs}
\usepackage{xcolor}
\usepackage{tikz}
\usepackage{pgf}

\makeatletter\@ifpackageloaded{underscore}{}{\usepackage[strings]{underscore}}\makeatother
\newtheorem{thm}{Theorem}[section]
\newtheorem{lem}[thm]{Lemma}
\newtheorem{prop}[thm]{Proposition}
\newtheorem{cor}[thm]{Corollary}

\newtheorem{rem}[thm]{Remark}

\definecolor{axeZ}{RGB}{20,81,204}
\definecolor{axeY}{RGB}{212,43,11}
\definecolor{axeX}{rgb}{0.33, 0.42, 0.18}
\definecolor{mygreen}{rgb}{0.20  0.43 0.22} 
\definecolor{mypurple}{rgb}{0.37 0 0.51}  

\definecolor{axisX}{rgb}{0.281, 0.722, 0.216}
\definecolor{axisY}{rgb}{0.613, 0.125, 0.961}
\definecolor{axisZ}{rgb}{0.785, 0.527, 0.047}
\definecolor{axisT}{rgb}{0.082, 0.402, 0.539}

\newcommand\perm[2]{\ensuremath{({\color{axeY}#1}, {\color{axeZ}#2})}}
\newcommand\bsig{{\ensuremath{\boldsymbol{\sigma}}}}

\newcommand\cT{{\ensuremath{\mathcal{T}}}}

\newcommand\bpi{{\boldsymbol{\pi}}}

\newcommand\treea{\ensuremath{\perm{21}{12}}}
\newcommand{\treeb}{231}
\newcommand{\treeset}{S^{d-1}_n\left(\treea,\treeb \right)}

\newcommand{\bfdir}{\textbf{dir}}

\usepackage{listings}
\usepackage{placeins}

\usepackage{color}

\definecolor{deepblue}{rgb}{0,0,0.5}
\definecolor{deepred}{rgb}{0.6,0,0}
\definecolor{deepgreen}{rgb}{0,0.5,0}
\DeclareFixedFont{\ttb}{T1}{txtt}{bx}{n}{9} 
\DeclareFixedFont{\ttm}{T1}{txtt}{m}{n}{9}  
\newcommand\pythonstyle{\lstset{
language=Python,
basicstyle=\ttm,
morekeywords={self},              
keywordstyle=\ttb\color{deepblue},
emph={MyClass,__init__},          
emphstyle=\ttb\color{deepred},    
stringstyle=\color{deepgreen},
frame=tb,                         
showstringspaces=false
}}
\lstnewenvironment{python}[1][]
{
\pythonstyle
\lstset{#1}
}
{}

\newcommand\pythoninline[1]{{\pythonstyle\lstinline!#1!}}

\usepackage[letterpaper,top=2cm,bottom=2cm,left=2.75cm,right=2.75cm,marginparwidth=1.75cm]{geometry}
\usepackage{kpfonts}

\newcommand{\minipdf}[2]{\begin{minipage}{#1\textwidth}\center{\resizebox{0.9\textwidth}{!}{\includegraphics{#2}}}\end{minipage}}
\usepackage{graphicx}
\usepackage[colorlinks=true, allcolors=blue]{hyperref}
\usepackage{algpseudocode}
\usepackage{algorithm}

\newtheorem{theorem}{Theorem}
\newtheorem{example}{Example}

\newtheorem{definition}[theorem]{Definition}

\title{Max-tree for $d$-permutations and pattern avoidance}
\author{Thomas Muller}

\begin{document}
\sloppy
\maketitle

\begin{abstract}

    Higher dimensional permutations are tuples of $d-1$ permutations that can be identified with a point set in a $d-$dimensional grid. In N. Bonichon and P.-J. Morel, {\it J. Integer Sequences} 25 (2022), several conjectures regarding the enumeration of pattern avoiding $d-$permutations were stated.

    In this paper, we consider a mapping from $d-$permutations to $2^{d-1}-$ary trees that naturally generalizes the classical max-tree construction for permutations. We then show that, when restricted to $d-$permutations avoiding $\treea$ and $\treeb$, this mapping yields a bijection with d-ary trees. This result resolves one of the conjectures of Bonichon and Morel.
    

\end{abstract}

\tableofcontents

\section{Introduction}

Permutations are bijections from the set $[n]=\{1,\ldots,n\} $ to itself that represent rearrangements of distinguishable elements. They are central objects in combinatorics that can be studied from many different perspectives. In particular, the notion of pattern avoidance in permutations has led to numerous connections with other combinatorial structures (see for instance \cite{Kitaev2011PatternsIP} and \cite{Bona}). One celebrated example of such a connection is the enumeration of permutations avoiding a single pattern of size three which is given by the Catalan numbers \cite{knuth1973art}. 

In this work, we focus on a classical mapping from permutations to binary trees  that we recall here: the \emph{max-tree} of a permutation. This tree is obtained through a recursive decomposition of the permutation around its maximal element. A permutation $\pi$ is viewed as a set of points $(i,\pi(i))$ in the plane and the root of the tree is defined as the point with maximal value. Removing this point partitions the permutation into two subpermutations: the points lying to its left and those lying to its right. Then, applying the same construction recursively to these two subpermutations yields the left and right subtrees of the root, respectively.

\begin{figure}[!htb]
        \center{\minipdf{0.48}{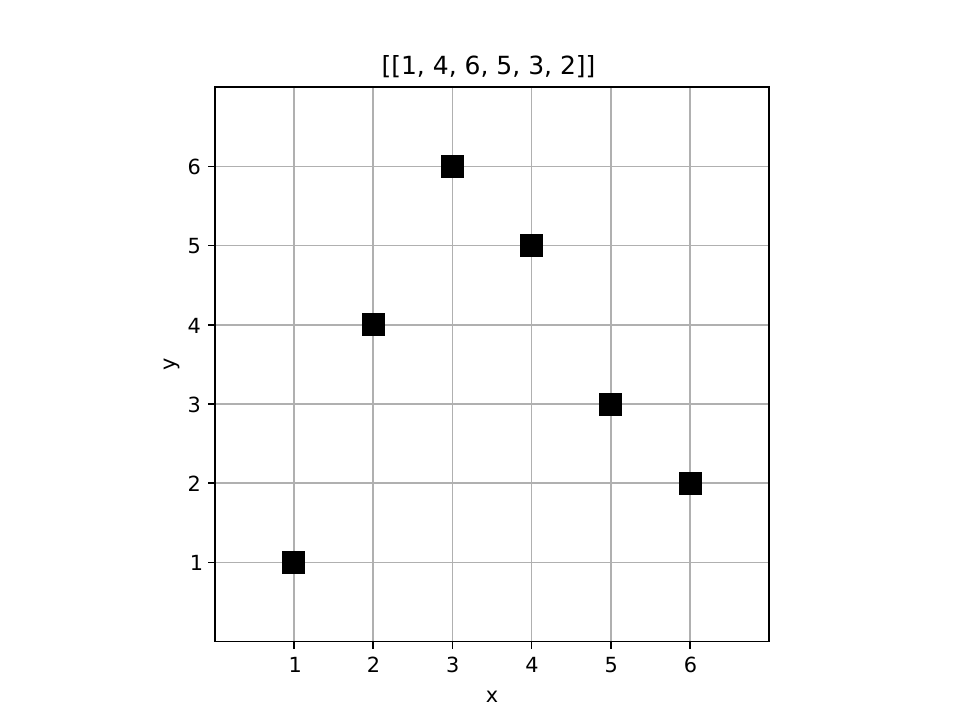} 	\minipdf{0.30}{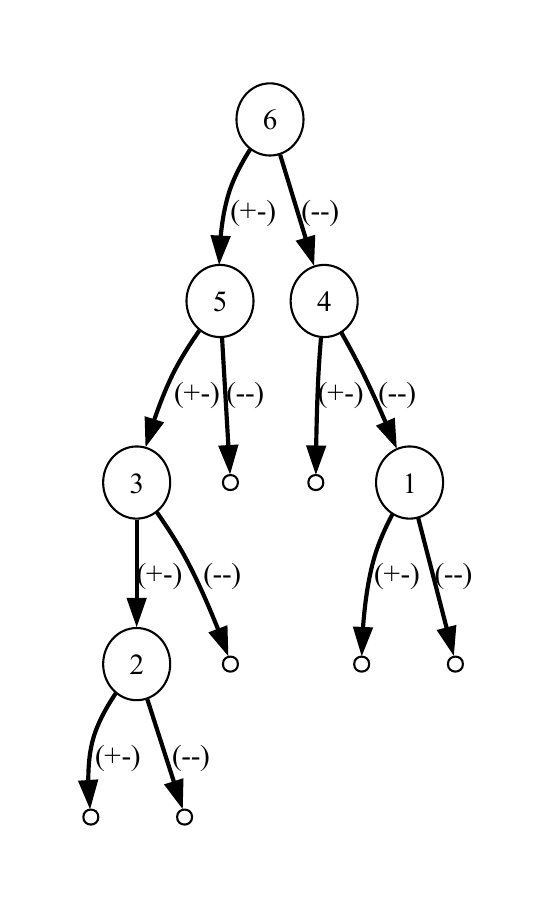}}
        \caption{A permutation and its max-tree}
        \label{fig:ex-tree-2d}
    \end{figure}

It is well-known that the max-tree construction induces a bijection between $231$-avoiding permutations and binary trees. Moreover, considering the dual construction, in which the minimal point is used instead of the maximal one, yields a second binary tree called the \emph{min-tree}. The pairs formed by the min-tree and the max-tree are \emph{twin binary trees}, and it is also known that the set of such pairs is in bijection with Baxter permutations \cite{DULUCQ199691}.

In this article, we generalize the notion of max-tree to the higher-dimensional setting of $d-$permutations. A \emph{$d-$permutation} is a tuple of $d-1$ permutations that can be geometrically represented as a set of points in a $d-$dimensional grid. While, in the two-dimensional case, the recursive decomposition of the max-tree is based on the left and right regions determined by a point, in the higher-dimensional setting it relies on the hyperoctants lying below a point with respect to the last axis of the grid. Consequently, the associated recursive structure is no longer a binary tree but a $2^{d-1}-$ary tree. 

Moreover, if one considers the sequence of points of a $d-$permutation $\bpi$ ordered by decreasing last coordinate and then projects these points onto the first $d-1$ coordinates, the $d-1$ dimensional hyperoctree (also called $2^{d-1}-$tree) \cite{treed} associated with this set of points corresponds to the max-tree of the $d-$permutation $\bpi$. The max-tree of a $3-$permutation thus corresponds to the quadtree \cite{quadtrees} of a set of two-dimensional points, while the max-tree of a $4-$permutations corresponds to the octree \cite{octrees} of a set of three-dimensional points.

Additionally, as in the two dimensional case, the mapping from $d-$permutations to $2^{d-1}-$ary tree defined by the max-tree is surjective. This leads to the definition of certain \emph{admissible classes} of $d-$permutations, for which this mapping becomes bijective. The definition of these classes rely on total orderings of the hyperoctants lying below a point of a $d-$permutation.

The second contribution of this paper focuses on the specialization of this mapping to pattern avoiding $d-$permutations. Following the notion of pattern avoidance considered in \cite{bonichon2022baxter}, we show that the set $\treeset$ of $d-$permutations avoiding $\treea$ and $\treeb$, forms a subset of one of the admissible classes of $d-$permutations, for which the associated max-trees correspond to $d-$ary trees. This yields a bijection between $\treeset$ and the set of $d-$ary trees with $n$ internal nodes.

This result is of twofold interest. First, it naturally generalizes the classical bijection between $231-$avoiding permutations and binary trees to arbitrary dimension $d$. Second, it provides an enumeration of the $d-$permutations in $\treeset$: 
\begin{equation}
        |\treeset|= \frac{1}{dn+1} \binom{dn+1}{n}\;.
    \end{equation}
This answers a conjecture originally formulated in \cite{bonichon2022baxter} and extended in \cite{muller2025study}. It is also worth mentioning that this bijection adds to other bijective results involving pattern-avoiding $d$-permutations and combinatorial structures such as triangle solitaire \cite{schabanel20253} and higher-dimensional floorplans \cite{bonichon2025}. 

A third interest of this bijection lies in its application to the random generation of $d-$permutations in $\treeset$. By generating uniformly a large random $d-$ary tree (using the encoding by Lukasiewicz words and the cycle lemma)  and by applying the mapping from trees to permutations of this bijection, one can generate large uniform random $d-$permutations in $\treeset$ as in Figure \ref{fig:random_231}.

\begin{figure}[!htb]
        \center{ \minipdf{0.40}{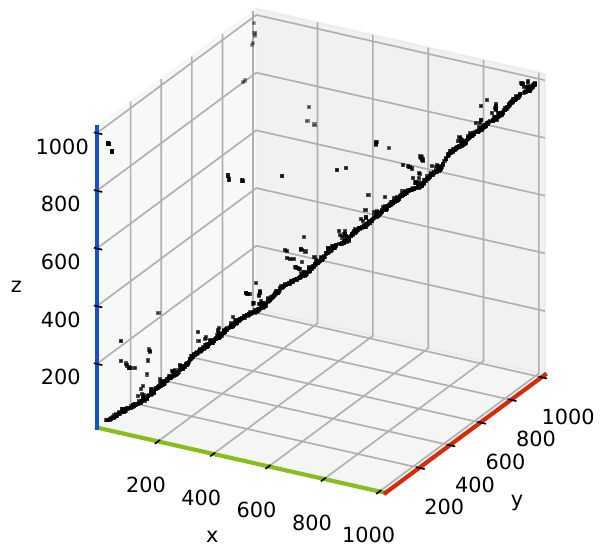} \minipdf{0.40}{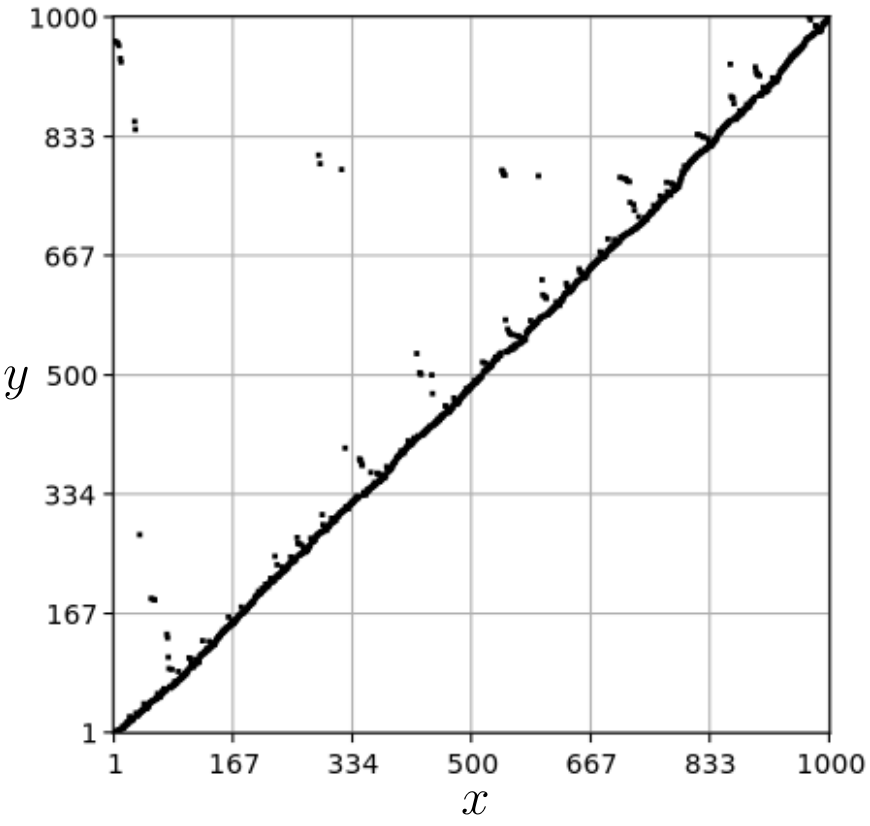} \minipdf{0.40}{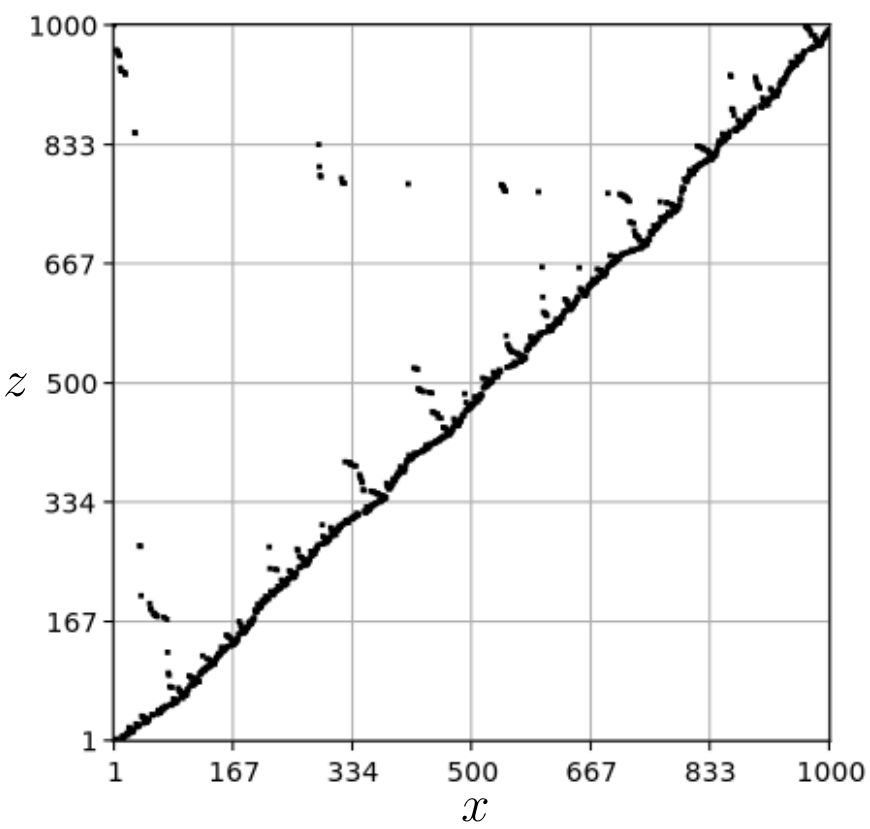} \minipdf{0.40}{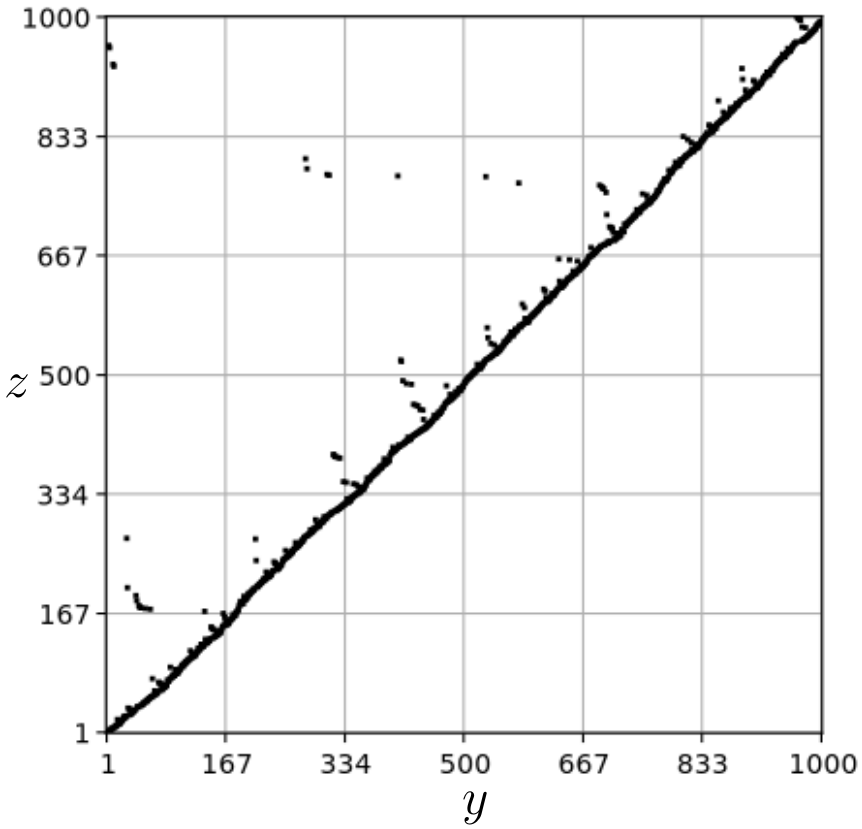}}
        \caption{A random $3-$permutation of size 1000 in $S^{2}_{n}(\treea,\treeb)$.}
        \label{fig:random_231}
    \end{figure}

This paper is organized as follows. In Section \ref{sec:section1}, we first recall basic notions on $d$-permutations and $d$-ary trees. We then introduce the max-tree associated with a $d$-permutation, together with admissible classes of $d$-permutations under this max-tree. We show that, restricted to these classes, the max-tree construction yields a bijection with $2^{d-1}$-ary trees. In Section \ref{sec:section2}, we first recall the notion of pattern avoidance used in this work and establish the bijection between $\treeset$ and the set of $d$-ary trees. We then conclude with several open problems related to this work in Section \ref{sec:Conclusion}.

\section{Prerequisites and notations}
\label{sec:intr}


In this subection, we recall some notions on $d-$permutations and $d-$ary trees that will be used in Section \ref{sec:section1} to introduce the mappings between $d-$permutations and $2^{d-1}-$ary trees. Note that in this section we do not consider the notion of pattern avoidance in $d-$permutations. This notion will be presented and used in Section \ref{sec:section2}.

\subsection*{$d-$permutations}

A $d-$permutation of size $n$ is a tuple of a $d-1$ permutations of size $n$.  We denote by $S^{d-1}_n$ the set of $d-$permutations of size $n$. Given a $d-$permutation $\bpi=(\pi_1, \ldots, \pi_{d-1})$, we also denote by $\pi_i$ the $i^{th}$ permutation in the tuple and by $\pi_i(j)$ the $j^{th}$ element of this permutation. The \emph{matrix representation} of $\bpi$ is the $d\times n$ matrix $$\bar{\bpi}= \begin{pmatrix}1 & 2 & \ldots & n-1 & n \\ \pi_1(1)& \pi_1(2) &\ldots& \pi_1(n-1)& \pi_1(n) \\ &  &\vdots& & \\ \pi_{d-1}(1)& \pi_{d-1}(2) &\ldots& \pi_{d-1}(n-1)& \pi_{d-1}(n) \end{pmatrix} \;.$$

The $d-$dimensional diagram of $\bpi$, denoted $\Pi(\bpi)$, is the point set $$\Pi(\bpi) := \{ (i,\pi_1(i), \ldots, \pi_{d-1}(i)), i \in [n] \} \;. $$

This point set is obtained by considering that each column in $\bar{\bpi}$ correspond to the position vector of a point in a $d-$dimensional grid. In the following, we consider that the coordinates of this point set are given by $x_0, \ldots ,x_{d-1}$. In the two and three dimensional case, we will also often denote by $x,y$ and $x,y,z$ those coordinates. Moreover, we will often abuse notation and denote by $\bpi(p)$ the position vector of the point $p$ in $\Pi(\bpi)$ and denote by $\pi_i(p)$ the coordinate $i$ of this position vector. 

\begin{figure}[!htb]
        \center{\minipdf{0.49}{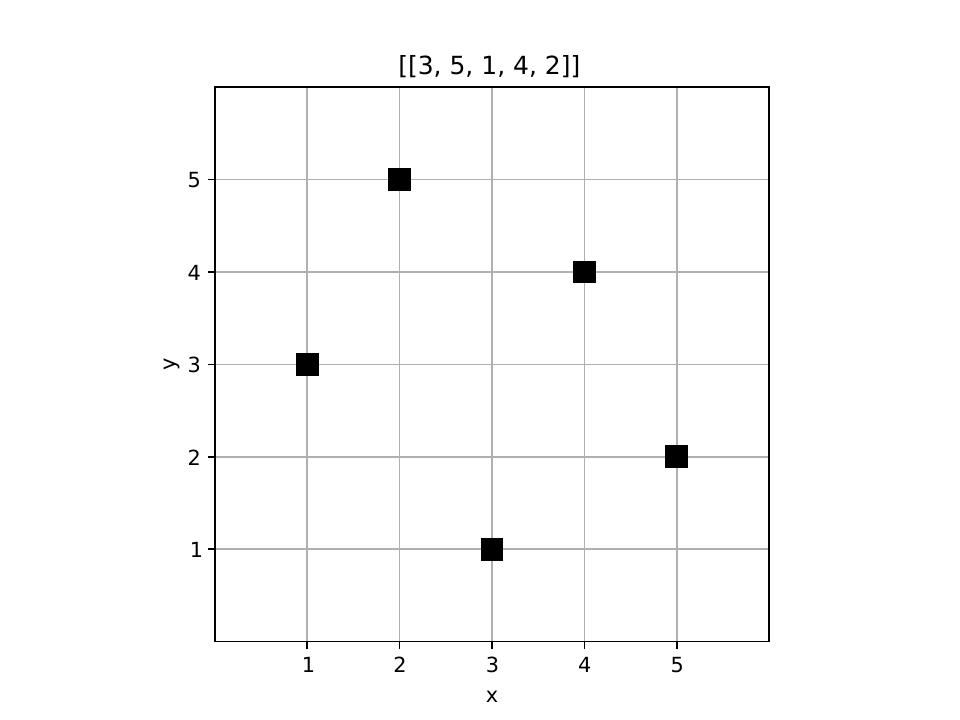} \minipdf{0.49}{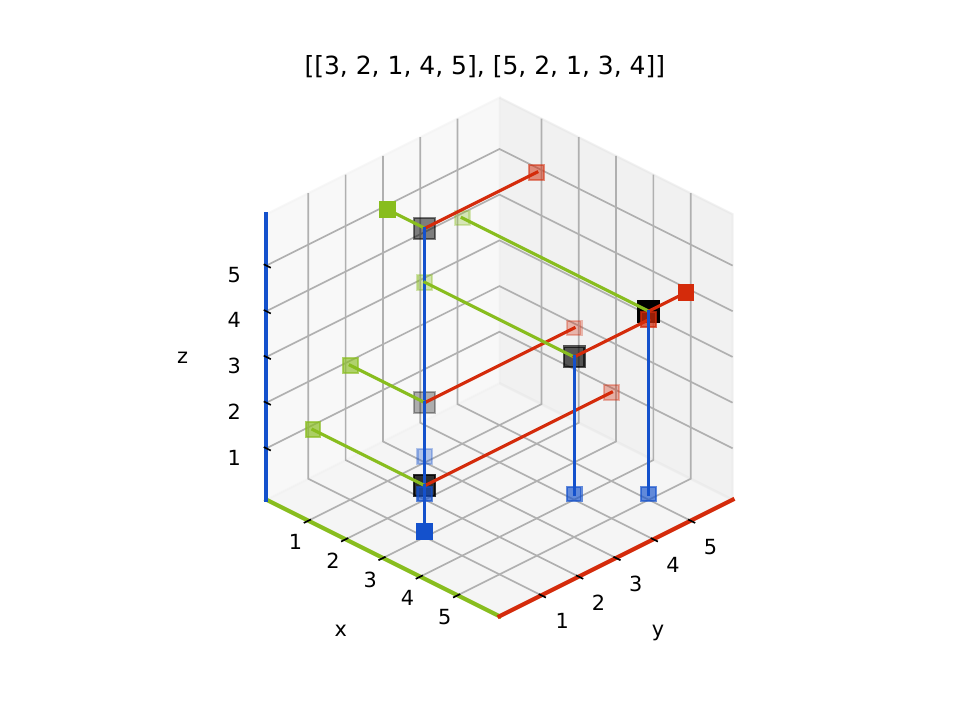}}
        \caption{A permutation and a $3-$permutation.}
        \label{fig:ex_tree}
    \end{figure}

\subsection*{Directions of dimension $d$}

A \emph{direction} ${\bf dir}$ is an element of $\{+1,-1\}^d$. A direction is \emph{negative} with respect to the axis $k$ if its $k^{th}$ element is $-1$. In a $d$-permutation, there are $2^{d}$ different directions and $2^{d-1}$ negative directions of each axis. In what follows, we focus on directions that are negative with respect to the last axis (whose coordinate is $x_{d-1}$), we denote this set (in dimension $d$) by $\mathbb{F}^d$. 

Let also ${\bf dir}$ be a direction, the \emph{opposite} of ${\bf dir}$, denoted $(-{\bf dir})$, is the direction such that $(-{\bf dir})=(-1)\times {\bf dir}$.
\begin{definition}
    Let $\bpi$ be a $d$-permutation with $n$ points and let $p_1$ resp. $p_2$ be two points in $\bpi$. The \textbf{direction} between the two points $(p_1,p_2)$, denoted $\textbf{dir}(p_1,p_2)$, is defined as the direction ${\bf dir}$ such that $\left(\text{sign}(\pi_0(p_2)-\pi_0(p_1)),\ldots, \text{sign}(\pi_{d-1} (p_2) - \pi_{d-1}(p_1))\right)={\bf dir}$.
    \label{def:direction}
\end{definition}

In a $d-$permutation, each point possesses $2^{d}$ hyperoctants that are each associated to a direction. For two points $p_a$ and $p_b$ one has ${\bf dir}(p_a,p_b)= {\bf f}$ if the point $p_b$ lies in the hyperoctant of $p_a$ associated with the direction {\bf f}. Note that if ${\bf dir}(p_a,p_b)= {\bf f}$, one also has that ${\bf dir}(p_b,p_a)= {\bf  -f}$. In order to characterize the directions between the points of a $d-$permutation, we will only consider the directions in $\mathbb{F}^d$, then always write ${\bf dir}(p_a,p_b)$ where $p_a$ is the point with the greater value of the coordinate $x_{d-1}$.

\begin{example}
Figure \ref{fig:directions} shows in two and three dimensions the quadrants and octants of the points of a permutation with their associated directions. For instance in three dimensions, if a second point $p_2$ lies in the yellow region of the first point $p_1$, one has ${\bf dir}(p_1,p_2)= (-,+,-)$.
\begin{figure}[!htb]
        \center{\minipdf{0.48}{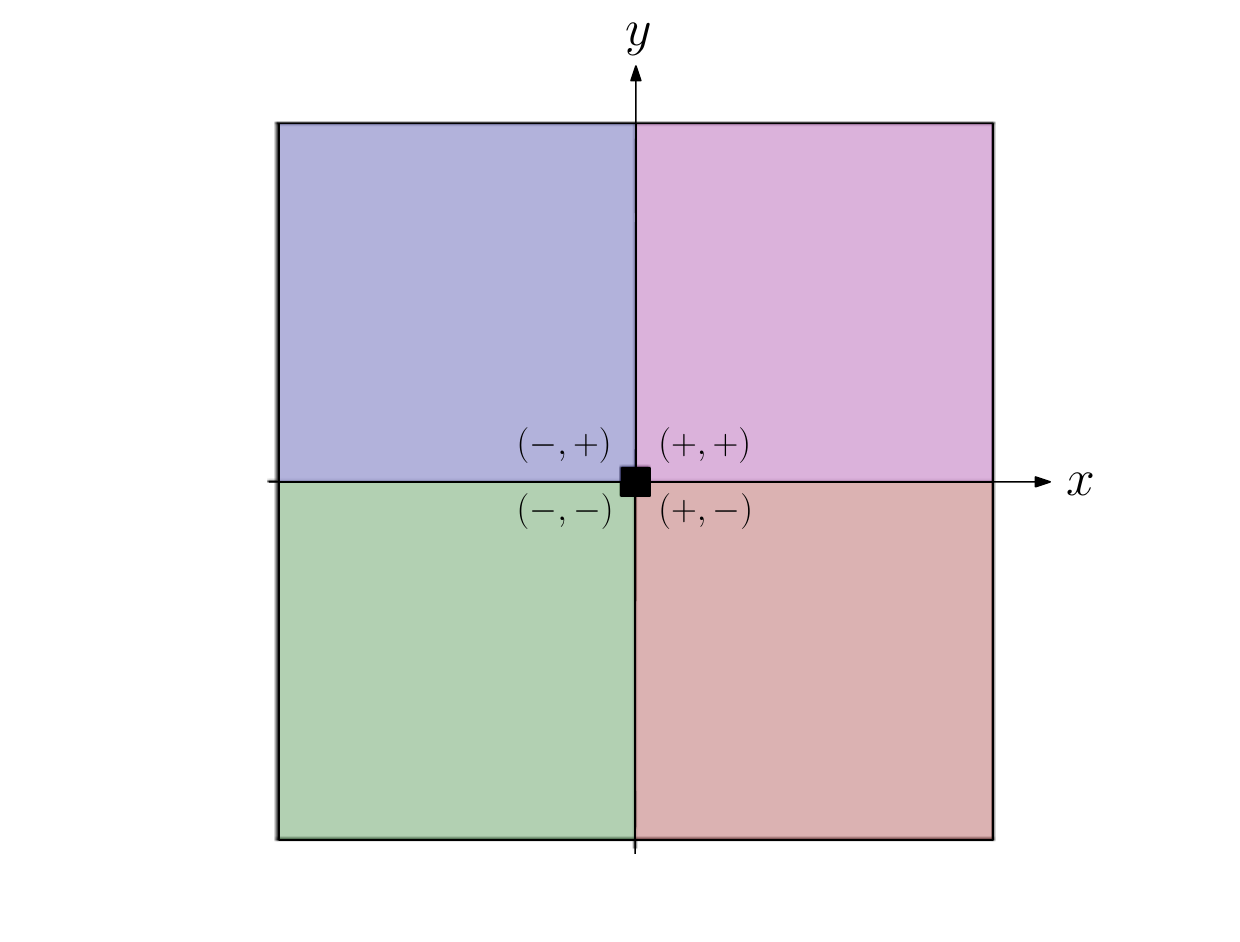} 	\minipdf{0.4}{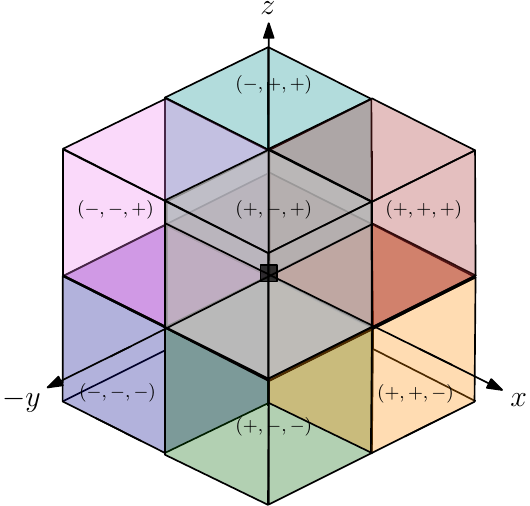}}
        \caption{On the left the quadrants of a point in a permutation. On the right the octants of a point in a $3-$permutation.}
        \label{fig:directions}
    \end{figure}
\end{example}

\subsection*{${d}$-ary trees}

A $d-$ary tree is a rooted tree in which all nodes have either $d$ distinguishable children or none. Nodes with no children are called leaves and the other are called internal nodes. 

In the following, we consider $2^{d-1}-$ary trees whose construction is governed by the directions in $\mathbb{F}^d$. Accordingly, we distinguish the children of each internal node by associating them with these directions. We denote by $\mathbb{T}^{2^{d-1}}_{n}$ the set of $2^{d-1}$-ary trees with $n$ internal nodes in which each child of a node is labeled by a direction in $\mathbb{F}^d$.

Given a tree $\cT$, we denote by $\cT(p)$ the subtree of $\cT$ whose root is the internal node $p$, we also denote by $\cT_{\bf dir}(p)$ the subtree defined by the child of direction ${\bf dir}$ of the internal node $p$. 

\begin{figure}[!htb]
        \center{\minipdf{0.7}{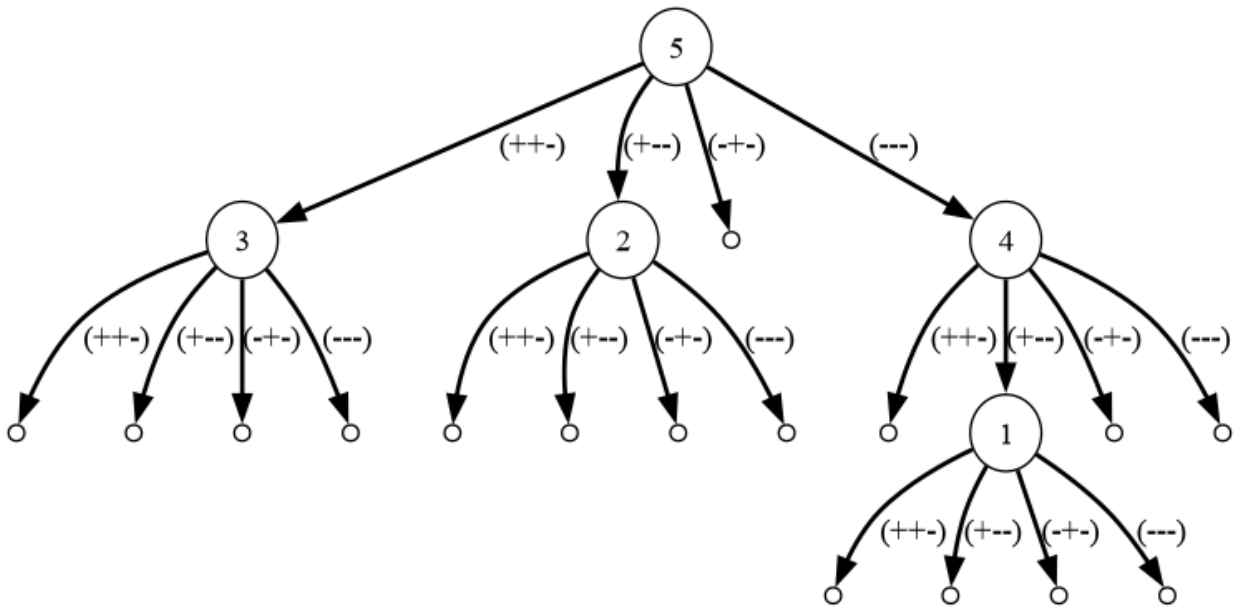}}
        \caption{A quaternary tree in $\mathbb{T}^{4}_5$}
        \label{fig:ex_tree}
    \end{figure}

\section{Mappings from $d-$permutations to $2^{d-1}-$ary trees}
\label{sec:section1}

\subsection{From permutations to trees}

Let $\bpi$ be a $d-$permutation of size $n$, the maximal point of $\bpi$ with respect to the coordinate $x_{d-1}$, denoted by $p_{max,d-1}$, is the point for which $\pi_k(p_{max,d-1})=n$. We also denote by $\bpi^{\bf f}$ the $d-$permutation defined by the set of points of $\bpi$ for which the direction with respect to $p_{max,d-1}$ is ${\bf f}$ (i.e all the points $p'$ of $\bpi$ for which ${\bf dir}(p_{max,d-1},p')= {\bf f}$). This splits $\bpi$ in $2^{d-1}$ subpermutations, each associated to a direction in $\mathbb{F}^d$. 
\begin{definition}
Let $\bpi$ be a $d$-permutation, the \emph{max-tree} of $\bpi$ with respect to the axis $d-1$, denoted $\cT_{max}(\bpi)$, is the $2^{d-1}$-ary tree obtained recursively such that:
\begin{itemize}
\item  The maximal point $p_{max,d-1}$ of $\bpi$ is the root of the tree.
\item It has $2^{d-1}$ subtrees: each such subtree is associated to a  direction $ {\bf f} \in \mathbb{F}^d$ and is given by  $\cT_{max}(\bpi^{\bf f})$. 
\item The recursion ends with leaves associated to the direction ${\bf dir} \in \mathbb{F}^d$ when the corresponding subpermutation associated with {\bf dir} is empty.  
\end{itemize} 
\label{def:perm2tree}
\end{definition}

The max-tree of a permutation with respect to another axis $k$ is defined similarly by considering the maximal point with respect to that axis and recursively applying the construction using the corresponding coordinate. Dually, one can define the min-tree of a d-permutation. In this case, the construction proceeds by choosing the minimal point with respect to a given axis and recursing upward rather than downward.

\begin{example}
Figure \ref{fig:ex-tree-3d} shows the $3$-permutation $\bpi=\perm{32145}{52134}$ and its max-tree with respect to the $z$ axis. The names of the internal nodes correspond to the position with respect to the $z$ axis of their corresponding point in the $3-$permutation.
\begin{figure}[!htb]
        \center{\minipdf{0.54}{figs/32145_52134.pdf} 	\minipdf{0.45}{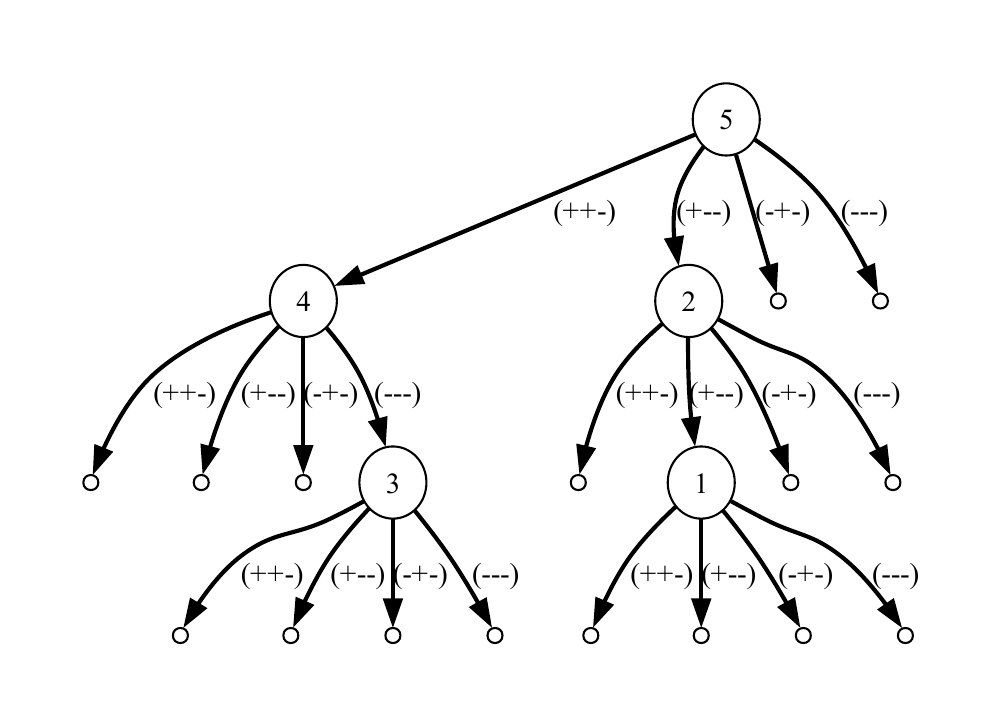}}
        \caption{A $3-$permutation and its max-tree with respect to the $z$ axis}
        \label{fig:ex-tree-3d}
    \end{figure}
    
Figure \ref{fig:ex-tree-2d} shows the $2$-permutation $\sigma=165243$ and its max-tree with respect to the axis $x$. The names of the internal nodes correspond to the position with respect to the coordinate $x$ of their corresponding point in the permutation. In the left part of the figure, the axes of the permutation are swapped to make the construction of the max-tree with respect to the $x-$axis easier to visualize.

\begin{figure}[!htb]
        \center{\minipdf{0.48}{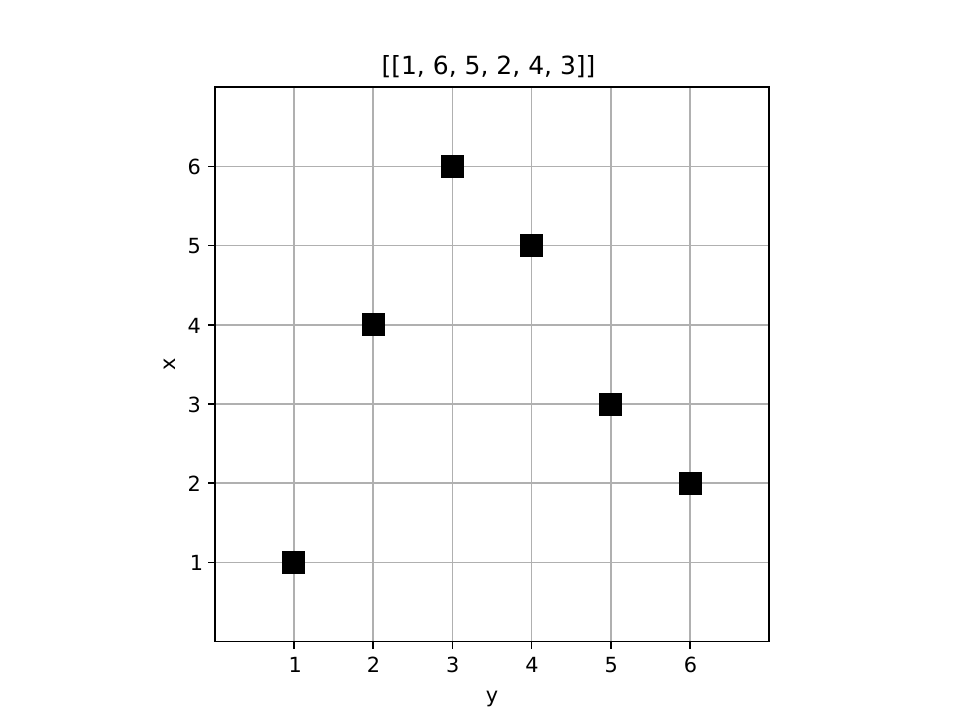} 	\minipdf{0.30}{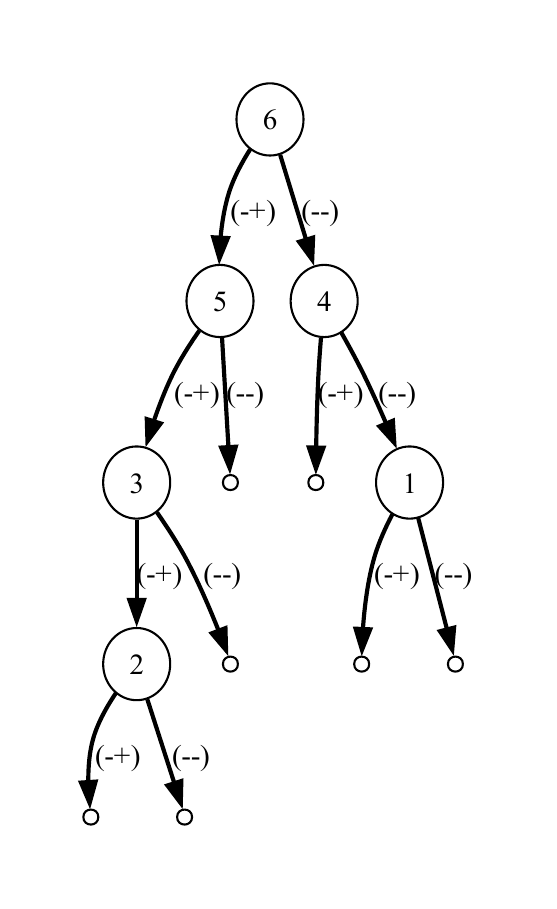}}
        \caption{A permutation and its max-tree with respect to the $x$ axis.}
        \label{fig:ex-tree-2d}
    \end{figure}
\end{example}

Definition \ref{def:perm2tree} yields a mapping $\gamma^d: S^{d-1}_n \to \mathbb{T}^{2^{d-1}}_{n}$ such that $\gamma^d(\bpi):=\cT_{max}(\bpi)$.

\begin{rem}
\label{rem:dir}
Each point of a $d-$permutation $\bpi$ corresponds to an internal node in its max-tree $\cT_{max}(\bpi)$. Let $p_a$ and $p_b$ be two points in $\bpi$ and let $r_a$ and $r_b$ be their corresponding internal nodes in $\cT_{max}(\bpi)$. Let also $r_{c_1} \ldots r_{c_m}$ be the ancestors of $r_b$ in $\cT$ and let $p_{c_1}$ \ldots $p_{c_m}$ be their corresponding points in $\bpi$. 

One has that $r_a \in  \cT_{\bf f}(r_b)$ with ${\bf f} \in \mathbb{F}_k$ if and only if {\bf dir}$(p_b,p_a)=\bf f$ and if for $i=1,\ldots,m$ one has that {\bf dir}$(p_{c_i},p_a)= {\bf dir} (p_{c_i},p_b)$.
\end{rem}

It can be shown by induction that the mapping $\gamma^d$ is surjective (see section \ref{sec:proof}), however this mapping is not injective. For example, the two permutations 41523 and 43512 are mapped to the same tree by $\gamma^2$. These two permutations and the corresponding tree are shown in Figure \ref{fig:ex-surjec}.

\begin{figure}[!htb]
        \center{ \minipdf{0.36}{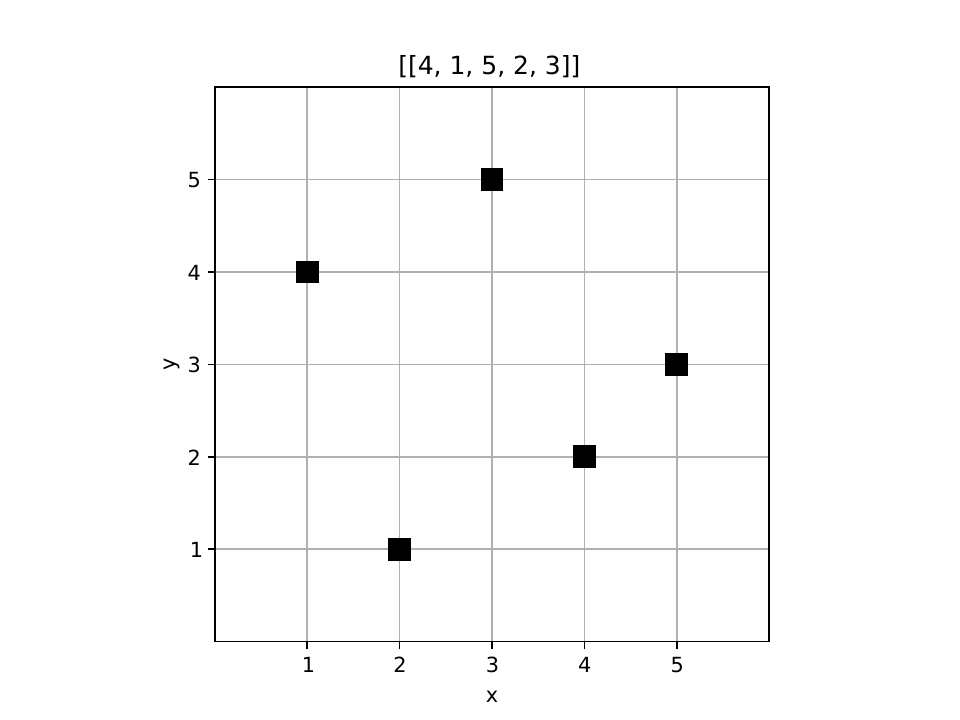} \minipdf{0.36}{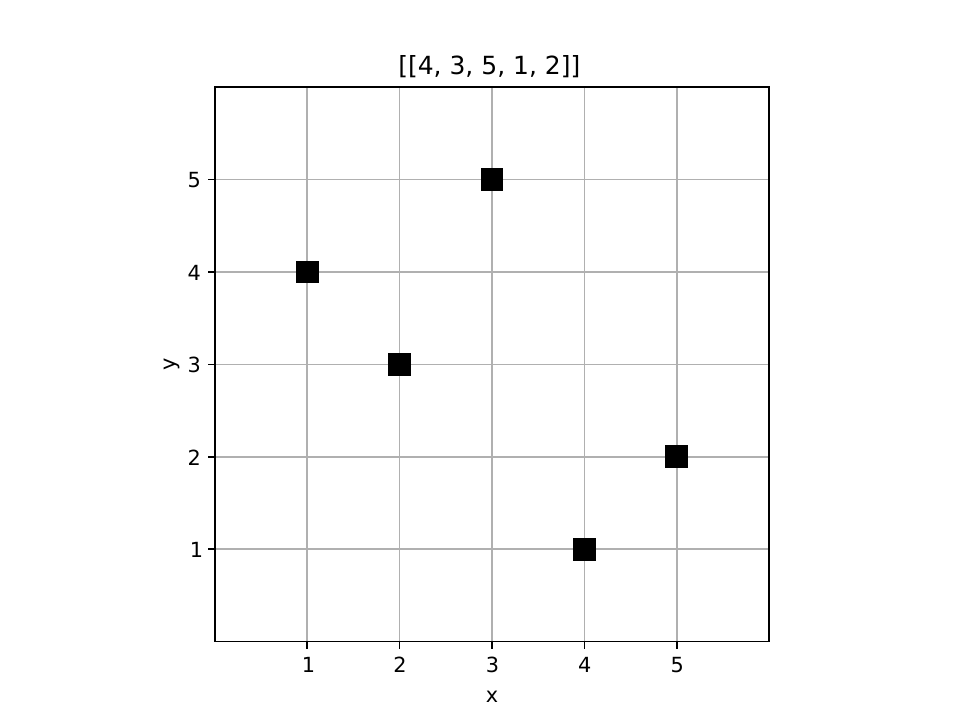} \minipdf{0.25}{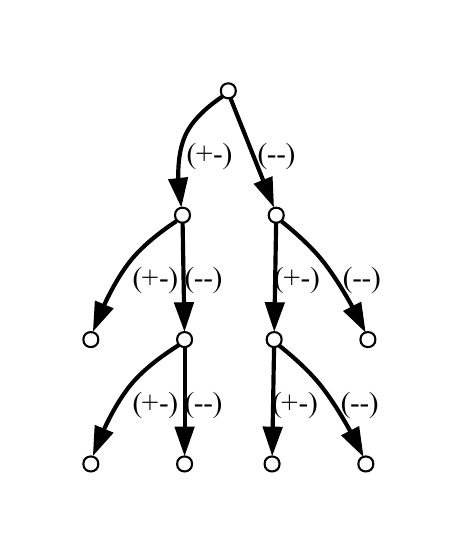}}
        \caption{Two permutations with the same max-tree with respect to the axis $y$.}
        \label{fig:ex-surjec}
    \end{figure}

\subsection{Compatible total orders and main theorem}

Let $C$ be a total order on a set of directions $F$. We say that $C$ is \emph{compatible} with respect to the axis $k$ if,  in $C$, there is no direction for which the $k^{th}$ coordinate is a $+$ that precedes a direction for which the $k^{th}$ coordinate is a $-$. Similarly, let ${\bf C}= (C_0,\ldots,C_{d-1})$ be a tuple of $d$ total orders on $F$, we say that {\bf C} is a compatible set of total orders if, for $i=1,\ldots,d$, $C_i$ is compatible with respect to the axis $i$.

\begin{definition}
    Let ${\bf C}=C_0,\ldots,C_{d-1}$ be a compatible set of total orders on $\mathbb{F}^d$. Let $\bpi$ be a $d$-permutation and let $\cT=\gamma^d(\bpi)$ be its max-tree with respect to the axis $d-1$. 

    We say that $\bpi$ is \emph{admissible} with respect to ${\bf C}$ if, for every internal node $r$ of $\cT$, the subpermutations of $\bpi$ corresponding to the subtrees of $\cT$ attached to $r$ form ordered blocks on each coordinate axis $l$, according to the order prescribed by $C_l$.
    
    More precisely, for every coordinate $l\in\{1,\ldots,d\}$ and for any two directions ${\bf f}_1,{\bf f}_2 \in \mathbb{F}^d$ such that ${\bf f}_1$ precedes ${\bf f}_2$ in $C_l$, every point corresponding to a node of the subtree $\cT_{{\bf f}_1}(r)$ has smaller $l$-th coordinate than every point corresponding to a node of the subtree $\cT_{{\bf f}_2}(r)$.
\label{def:canonical}
\end{definition}

We denote by $S^{d-1}_{n,{\bf C}}$ the set of $d-$permutations of size $n$ that is admissible with respect to ${\bf C}$. We also denote the restriction of $\gamma^d$ to $S^{d-1}_{n,{\bf C}}$ by $\gamma^d_{\bf C}$.

\begin{example}
    Let us consider ${\bf C_{ex}}= C_0, C_1, C_2$ the compatible set of total orders on $\mathbb{F}^3$ defined such that:
    $$C_0: (-,-,-)<(-,+,-)<(+,-,-)<(+,+,-)\,,$$
    $$C_1: (-,-,-)<(+,-,-)<(-,+,-)<(+,+,-)\,,$$
    $$C_2: (-,-,-)<(+,-,-) < (-,+,-) <(+,+,-)\,.$$

    A $3-$permutations that is admissible with respect to ${\bf C}$ has the general form given in Figure \ref{fig:ex-perm_ord}, where each $\bpi^f$ (with $f \in \mathbb{F}^d$) has the same structure recursively. 

    \begin{figure}[!htb]
        \center{ \minipdf{0.40}{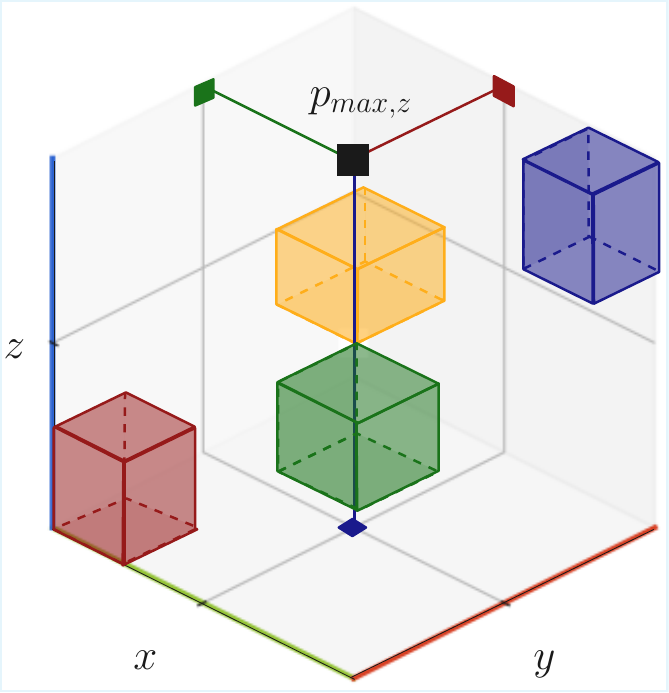} \minipdf{0.40}{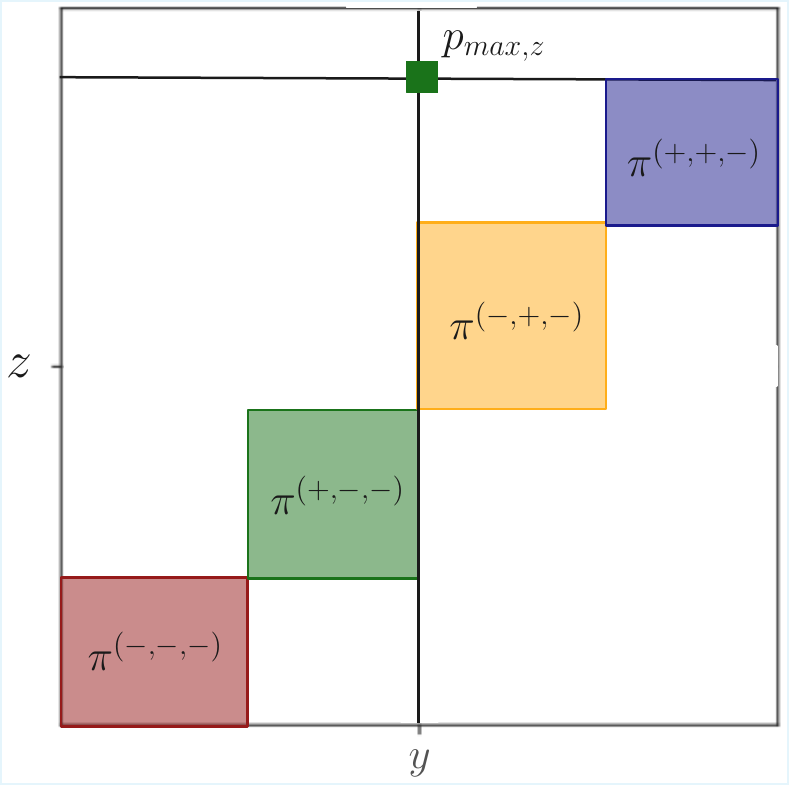} \minipdf{0.40}{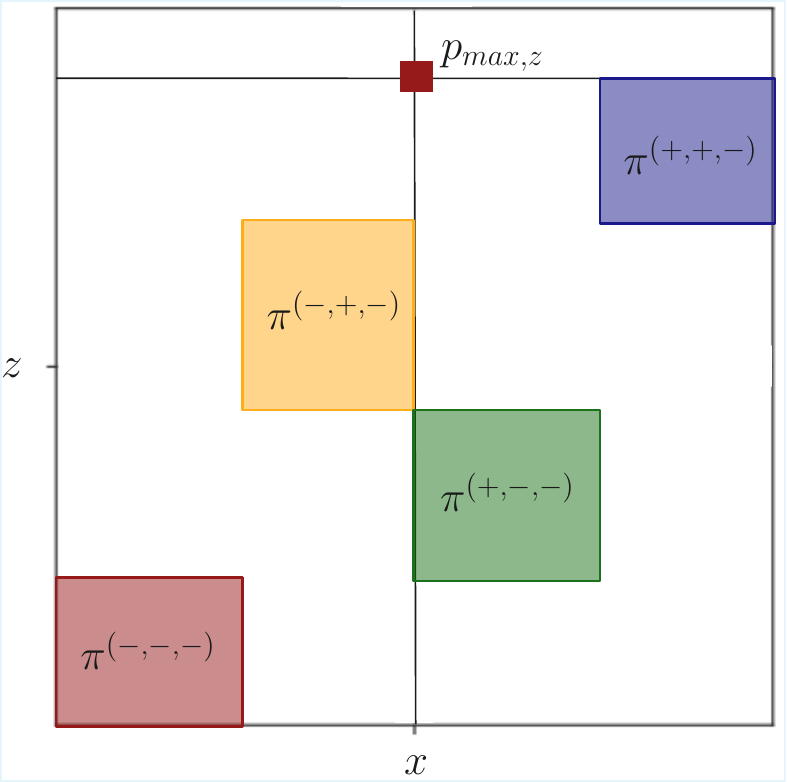} \minipdf{0.40}{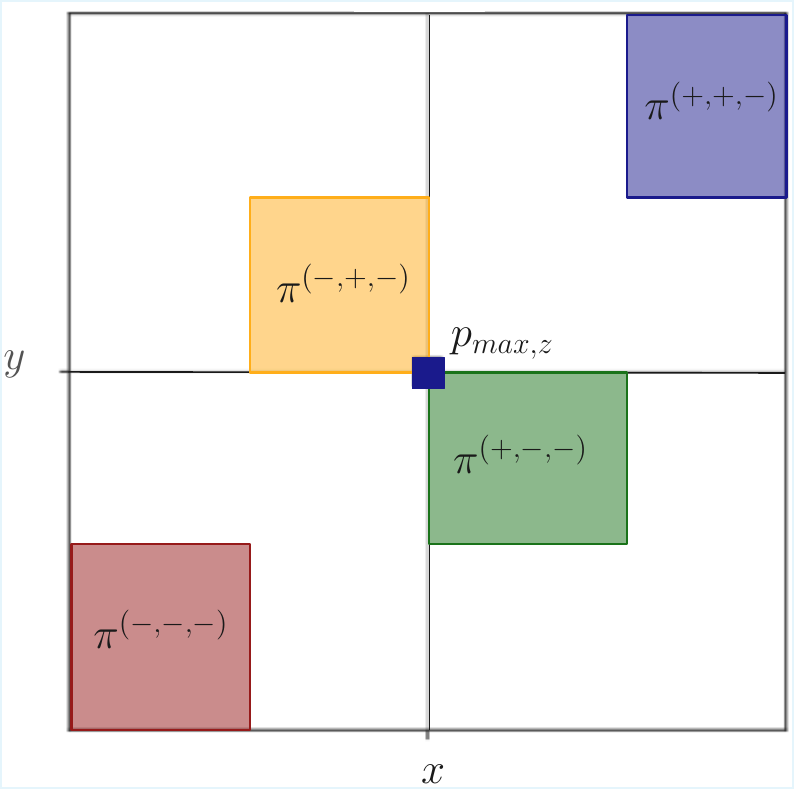}}
        \caption{The structure of a $3-$permutation respecting ${\bf C_{ex}}$.}
        \label{fig:ex-perm_ord}
    \end{figure}

\end{example}

 In the next subsection, we show that the mapping $\gamma^d$ is surjective on $\mathbb{T}^{2^{d-1}}_{n}$. Then, we show that for any compatible set of total order ${\bf C}$, the fiber of each element of $\mathbb{T}^{2^{d-1}}_{n}$  under $\gamma^d$ contains a unique $d-$permutation that is admissible with respect to ${\bf C}$. This leads to the following theorem.
 
\begin{thm}
    For any compatible set of total order ${\bf C}$, the set $S^{d-1}_{n,{\bf C}}$ is in bijection with $\mathbb{T}^{2^{d-1}}_{n}$.
    \label{thm:main1}
\end{thm}

Additionally, this bijection can be be naturally extended to trees whose arity is not given by a power of $2$ (through another restriction).

Let $F$ be a subset of $\mathbb{F}^d$ and let ${\bf C}_F$ be a compatible set of total orders on $F$. Let also $\bpi$ be a $d-$permutation  and $\cT$ be its max-tree  with respect to the axis $d-1$. We say that $\bpi$ is admissible with respect to ${\bf C}_F$ if it satisfies the condition of Definition \ref{def:canonical} and if there are no two internal nodes $r_1$ and $r_2$ in $\cT$ such that for any direction ${\bf f} \notin F$ one has $r_1 \in \cT_{\bf f}(r_2)$. 

Let us also consider $\mathbb{T}^{2^{d-1}}_{n}(F)$, the subset of $\mathbb{T}^{2^{d-1}}_{n}$ defined such that
$$\mathbb{T}^{2^{d-1}}_{n}(F):=\{ \cT \in \mathbb{T}^{2^{d-1}}_{n}: \text{there are only internal nodes of direction } {\bf f} \in F \text{ in } \cT \} \;.$$ It is clear that the trees in $\mathbb{T}^{2^{d-1}}_{n}(F)$ are the max-trees of the permutations that are admissible with respect to ${\bf C}_F$. There is thus a bijection between $S^{d-1}_{n,{\bf C}_F}$ and $\mathbb{T}^{2^{d-1}}_{n,d-1}(F)$.

Additionally, if $F$ is a subset of $\mathbb{F}_k$ containing $k$ directions, there is a bijection between the trees in $\mathbb{T}^{2^{d-1}}_{n}(F)$ and the set of $k-$ary trees with $n$ internal nodes: Let $\cT$ be a tree in $\mathbb{T}^{2^{d-1}}_{n}(F)$, for all directions ${\bf f} \in \mathbb{F}^d$ such that $ {\bf f} \notin F$, there are no internal nodes of direction ${\bf f}$ in $\cT$. One can thus remove all leaves of those directions in $\cT$ to obtain a $k-$ary tree. For example, let us consider the quaternary tree on the left of Figure \ref{fig:ex-ternary_bijec}. This tree has no internal node of direction $(+,+,-)$, one can thus remove all leaves of this direction to obtain the ternary tree shown on the right of Figure \ref{fig:ex-ternary_bijec}. Note also that one can do the inverse operation and start with the ternary tree and add a leaf of direction $(+,+,-)$ to each of its internal nodes to obtain the quaternary tree.  

\begin{figure}[!htb]
        \center{\minipdf{0.49}{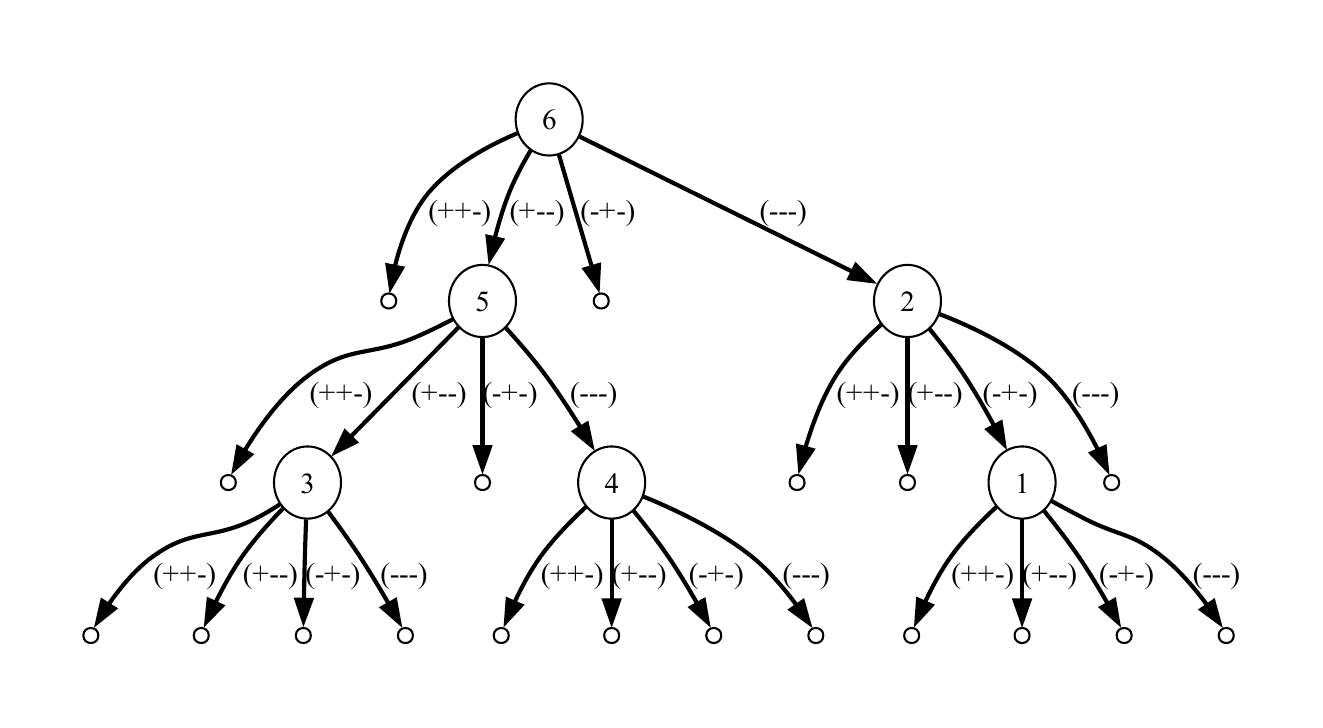} \minipdf{0.49}{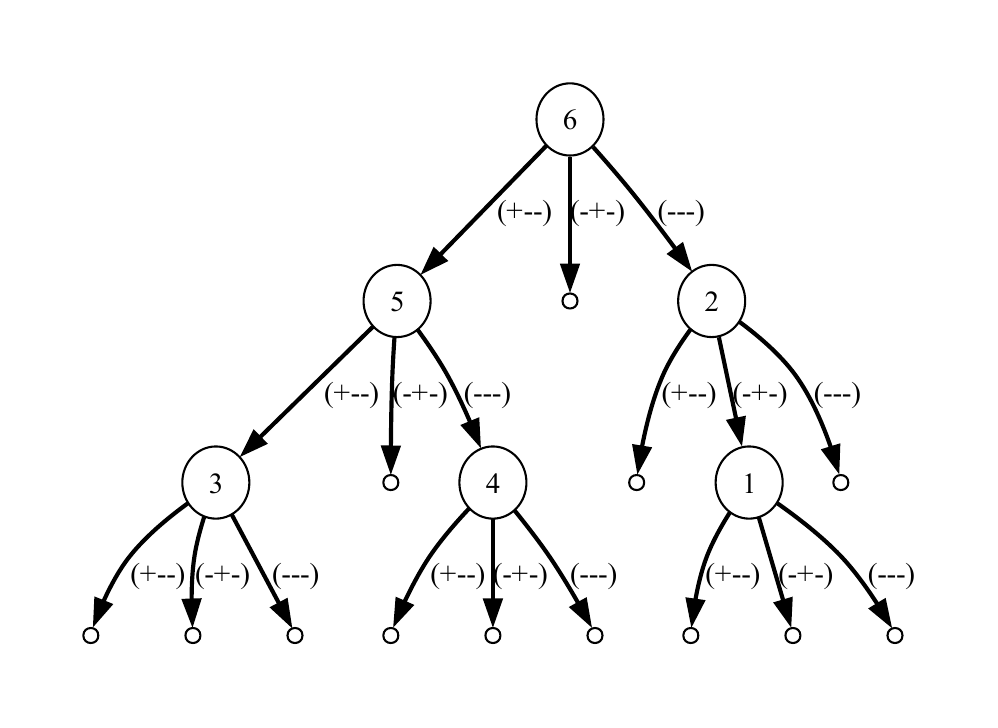}}
        \caption{On the left a quaternary tree. On the right the ternary tree obtained from it by removing the leaves of direction $(++-)$. }
        \label{fig:ex-ternary_bijec}
    \end{figure}

\begin{cor}
    Let $F$ be a subset of $\mathbb{F}^d$ containing $k$ directions and let ${\bf C}_F$ be a set of $d$ compatible total orders on $F$. The set of permutations $S^{d-1}_{n,{\bf C}_F}$ is in bijection with the set of $k-$ary trees.
    \label{cor:bijec}
\end{cor}

\subsection{Proof of the main theorem}
\label{sec:proof}

Let us now prove Theorem \ref{thm:main1} in this subsection. We first prove that $\gamma^d_{\bf C}$ is surjective on $\mathbb{T}^{2^{d-1}}_{n}$ (which implies that $\gamma^d$ is also surjective on $\mathbb{T}^{2^{d-1}}_{n}$). Then, we prove that $\gamma^d_{\bf C}$ is injective.

\begin{lem}
    The mapping $\gamma^d_{\bf C}$ is surjective.
\label{lem:injective}
\end{lem}
\begin{proof}
We show this by induction on the size of the $d-$permutations. At size 0, the only element of $\mathbb{T}^{2^{d-1}}_{0}$ is a single leaf which is the max tree of the empty $d-$permutation in $S^{d-1}_{0,{\bf C}}$. Similarly at size $1$, the only element of $\mathbb{T}^{2^{d-1}}_{1}$ is a single internal node with $2^{d-1}$ leaves, which is also the max-tree of the $d-$permutation with one point in $S^{d-1}_{1,{\bf C}}$.

Let us now consider a tree in $ \cT \in \mathbb{T}^{2^{d-1}}_{k}$. It consists of a root node together with $s$ subtrees $\cT_1, \ldots,\cT_s$ attached to this root. Without loss of generality, assume that each subtree $\cT_i$ is associated to a direction ${\bf dir}_i \in \mathbb{F}^d$  and contains $k_i$ internal nodes with $k_i <k$.

By the induction hypothesis, each subtree $\cT_i$ is the max-tree of a $d-$permutation ${\bsig}^{{\bf dir}_i}$. It follows that $\cT$ is the max-tree of any $d-$permutation $\bpi$ whose subpermutations with respect to the maximal point $p_{max,d-1}$ satisfy, for every ${\bf f} \in \mathbb{F}^d$, $\bpi^{\bf f}= {\bsig}^{\bf f}$. Moreover, for such a permutation to belong to $S^{d-1}_{k,{\bf C}}$ it suffices that the subpermutations $\bpi^{\bf f}$appear as ordered blocks according to the total orders prescribed by {\bf C} on each coordinate axis, as described in Definition~\ref{def:canonical}. Since {\bf C} is a compatible set of total orders such an arrangement is always possible. Therefore, there always exists at least one $d-$permutation $\bpi \in S^{d-1}_{k,{\bf C}}$ such that $\gamma^d_{\bf C}(\bpi)=\cT$.  
\end{proof}

\begin{lem}
    For any tree $\cT \in \mathbb{T}^{2^{d-1}}_{n}$ and for any compatible set of total order {\bf C}, there is no more than one $d-$permutation that is admissible with respect to {\bf C} in the fiber of $\cT$ under $\gamma^d$.
\label{lem:injective}
\end{lem}
\begin{proof}
    Let us consider two $d-$permutations $\bpi_A$ and $\bpi_B$ such that $\gamma^d(\bpi_A)= \gamma^d(\bpi_B)= \cT$. To prove the above lemma, we show that if $\bpi_A$ respects {\bf C} and $\bpi_A \neq \bpi_B $, then $\bpi_B$ cannot respect {\bf C}.

    If $\gamma^d(\bpi_A)= \gamma^d(\bpi_B)= \cT$ and $\bpi_A \neq \bpi_B $, there must exist two internal nodes of $\cT$, that we denote by $r_1$ and $r_2$, such that their corresponding points in $\bpi_A$ (denoted by $p_{A,1}$ and $p_{A,2}$) and the ones in $\bpi_B$ (denoted by $p_{B,1}$ and $p_{B,2}$) satisfy:

    \begin{equation}
    \label{eq:inj}
    \begin{split}
    \pi_{A,l}(p_{A,1}) &< \pi_{A,l}(p_{A,2}) \;, \\ \pi_{B,l}(p_{B,2}) &< \pi_{B,l}(p_{B,1})  \;.
    \end{split}
    \end{equation}
Equation \eqref{eq:inj} implies that $r_1$ cannot be an ancestor of $r_2$ in $\cT$ or the inverse. If $r_1$ was an ancestor of $r_2$, there would be a direction ${\bf f} \in \mathbb{F}_k$ such that $r_2 \in \cT_{\bf f}(r_1)$. However, by Equation \eqref{eq:inj} one cannot have both {\bf dir}$(p_{A,1},p_{A,2})= {\bf f}$ and {\bf dir}$(p_{B,1},p_{B,2})= {\bf f}$. A similar statement holds if $r_2$ is an ancestor of $r_1$.

If $r_1$ and $r_2$ are not in an ancestor-descendant relation in $\cT$, it implies that there must be a third internal node $r_3$ in $\cT$ and two directions ${\bf f_1}$ and ${\bf f_2}$ in $\mathbb{F}_k$ such that  $r_1 \in \cT_{\bf f_1}(r_3)$ and $r_2 \in \cT_{\bf f_2}(r_3)$. Let $p_{A,3}$ and $p_{B,3}$ be the points of $\bpi_{A}$ and $\bpi_{B}$ that correspond to $r_3$, one has then:

\begin{equation}
    \label{eq:inj2}
    \begin{split}
    {\bf dir}(p_{A,3},p_{A,1})= {\bf dir}(p_{B,3},p_{B,1})= {\bf f_1} \;, \\ {\bf dir}(p_{A,3},p_{A,2})= {\bf dir}(p_{B,3},p_{B,2})= {\bf f_2}  \;.
    \end{split}
    \end{equation}
Additionally, if $\bpi_{A}$ respects ${\bf C}$,  one has by the top line of Equation \eqref{eq:inj} that ${\bf f_1}$ precedes ${\bf f_2}$ in $C_l$. However, by the bottom line of Equation \eqref{eq:inj} and by Equation \eqref{eq:inj2}, $\bpi_2$ cannot respect  {\bf C}.
\end{proof}

 This finishes the proof of Theorem \ref{thm:main1}.

\section{Pattern avoiding $d-$permutations and $d-$ary trees}
\label{sec:section2}

In this section, we consider a class of $d-$permutations defined by forbidden patterns and we show that this set is in bijection with $d-$ary trees. 

We first recall some useful definitions on pattern avoidance in $d-$permutations, then define the bijection that gives the main theorem of this paper and prove this bijection.

\subsection{Pattern avoidance and main theorem}

Various definition of pattern avoidance in $d-$permutations have been considered in \cite{AVGUSTINOVICH2024105801,chen2024}. In this work, we use the definitions from~\cite{bonichon2022baxter}.

\begin{definition}
    Let $\textbf{i}=i_1 \dots i_{d'}$ be a sequence of indices in $\{0,\ldots,d\}$, let also $\bsig=(\sigma_1, \dots , \sigma_{d-1})$ be a $d$-permutation of size $n$. The \emph{projection} on $\textbf{i}$ of $\bsig$ is the $d'-$permutation given by $\text{proj}_{\bf i}(\bsig) := (\sigma_{i_2} \sigma_{i_1}^{-1},\dots, \sigma_{i_{d'}} \sigma_{i_1}^{-1}) $. A projection is \emph{direct} if $i_1 < i_2 < \dots < i_{d'}$.
    \label{def:proj}
\end{definition}

\begin{definition}
    Let the $d$-permutation $\bsig = (\sigma_1, \dots , \sigma_{d-1}) \in S^{d-1}_n$ and the $d'-$permutation $\bpi = (\pi_1,\ldots , \pi_{d'-1}) \in S^{d'-1}_k$ with $k \leq n$ and $d'\leq d$. Then \emph{$\bsig$ contains the pattern $\bpi$} if there exists a direct projection $\bsig' = \text{proj}_i(\bsig)$ of dimension $d'$ and a set of indices $c_1 < \ldots < c_k$ s.t. $\sigma'_i(c_1) \ldots \sigma'_i(c_k)$ is order-isomorphic to $\pi_i$ for all $i$. A permutation avoids a pattern if it doesn’t contain it.
    \label{def:avoidance}
\end{definition}

\begin{example}
    Consider the patterns $k_1=\perm{132}{213}$, and $k_2=231$. The $3$-permutation $\bpi=\perm{1432}{3124}$ contains an occurrence of $k_1$, given by the $1^{st},3^{rd}$ and $4^{th}$ points of $\bpi$. As $\text{proj}_{1,2}(\bpi)= 3421$, $\bpi$ also contains an occurrence of $k_2$.
\end{example}

Given a set of patterns $ \bsig_1, \ldots, \bsig_m$, we denote by $S^{d-1}_n( \bsig_1, \ldots, \bsig_m)$ the set of $d-$permutations of size $n$ that avoid these patterns. We now state the main theorem of this paper.

\begin{thm}
    The permutations in $S^{d-1}_n(\treea,\treeb)$ are in bijection with the set of $d-$ary trees with $n$ internal nodes.
    \label{thm:main}
\end{thm}

Let us consider $F:=\{\bfdir^0,\ldots, \bfdir^{d-1} \}$, the subset of $\mathbb{F}^d$ defined such that 
\begin{equation}
    \bfdir^i =  \big( \underbrace{+,\ldots,+}_{i}, \underbrace{-,\ldots,-}_{d-i} \; \big) \; .
\label{eq:dirf}
\end{equation} 
On this subset, we consider the compatible set of total orders ${\bf C}_F= C_{F,0},\ldots, C_{F,d-1}$ defined such that:
\begin{itemize}
\item $C_{F,0} = C_{F,1} = \ldots = C_{F,d-1}$ and for two directions ${\bf dir^i}$, ${\bf dir^j}$ in $F$, ${\bf dir^i}$ precedes ${\bf dir^j}$ in $C_l$ if $i<j$.
\end{itemize}

We recall that $\mathbb{T}^{2^{d-1}}_{n}(F)$ is the subset of $\mathbb{T}^{2^{d-1}}_{n}$ defined such that:
$$\mathbb{T}^{2^{d-1}}_{n}(F):=\{ \cT \in \mathbb{T}^{2^{d-1}}_{n}(F): \text{there are only internal nodes of direction } {\bf f} \in F \text{ in } \cT \}$$.


In the next subsection, we will prove Theorem \ref{thm:main} by showing that the permutations of $\treeset$ are the permutations that respect {\bf $C_F$}. As explained in the previous section, as $F$ is a subset of $\mathbb{F}^d$ containing $d$ directions, the $d-$permutations of $\treeset$ are mapped under $\gamma^d$  (in a one to one correspondance) to the trees of $\mathbb{T}^{2^{d-1}}_{n}(F)$ which are themselves in a one to one correspondance with $d-$ary trees.    


Theorem \ref{thm:main} answer the conjectures of \cite{bonichon2022baxter} and \cite{muller2025study} in the affirmative. It also extends the famous bijection between the permutations avoiding a single pattern of size three, and binary trees \cite{knuth1973art}. However, the only $d-$permutation class that seems to have the same enumeration as $S^{d-1}_n(\treea,\treeb)$ is $S^{d-1}_n(\treea,312)$ this stands in constrats with the two dimensional case, where all permutation classes avoiding a pattern of size three are equinumerous. Note also that in arbitrary dimension, the permutations in the bijection presented in this paper are described only by two patterns, one of size two and dimension three and one of size three and dimension two. The simplicity of the characterization by forbidden patterns of the $d-$permutations in this bijection is thus remarkable. 

Since the enumeration of $d-$ary trees is well known, Theorem \ref{thm:main} leads to the following corollary:
\begin{cor}
    Let $|\treeset|$ be the numbers of permutations in $\treeset$, one has
    \begin{equation}
        |\treeset|= \frac{1}{dn+1} \binom{dn+1}{n}\;.
    \end{equation}
\end{cor}
\subsection{Proof of the bijection}

Let us now prove that the permutations of $\treeset$ correspond to the permutations respecting  ${\bf C}_F$. Once this is done, the bijection is then proven by Corollary \ref{cor:bijec}.

\begin{lem}
Let $\bpi$ be a $d$-permutation that avoids $\treea$ and let $p_1$ and $p_2$ be two points in $\bpi$ such that $\pi_{d-1}(p_1)> \pi_{d-1}(p_2)$, one has $\bfdir (p_1,p_2) \in F$.
\label{lem:dir}
\end{lem}
\begin{proof}
Suppose $\bfdir (p_1,p_2)$ is not in $F$, then either :
\begin{align*}
    &\text{1. }\bfdir (p_1,p_2)= \big( \underbrace{+,\ldots,+}_{i-j-1},-, \underbrace{+,\ldots,+}_{j},  \underbrace{-,\ldots,-}_{d-i} \; \big) \;, \quad \text{with $j+1\leq i \leq d-1$ and $1 \leq j$} \;, \\
    &\text{2. }\bfdir (p_1,p_2)= \big( \underbrace{+,\ldots,+}_{i},  \underbrace{-,\ldots,-}_{j},+,\underbrace{-,\ldots,-}_{d-i-j-1} \; \big) \;, \quad \text{with $0 \leq i$ and $0 < j<d-i-1$} \;.
\end{align*}
In the first case, one has:
\begin{align*}
    \pi_{0}(p_1) &< \pi_{0}(p_2) \;, \\ \pi_{i-j-1}(p_2) &< \pi_{i-j-1}(p_1)  \;, \\ \pi_{i-j}(p_1) &< \pi_{i-j}(p_2) \;, 
\end{align*}
In the second case, one has:
\begin{align*}
    \pi_{0}(p_1) &< \pi_{0}(p_2) \;, \\ \pi_{i}(p_2) &< \pi_{i}(p_1)  \;, \\ \pi_{i+j}(p_1) &< \pi_{i+j}(p_2) \;, 
\end{align*}
These two configurations correspond to an occurrence of $\treea$ in $\bpi$.
\end{proof}

Lemma \ref{lem:dir} implies that for any $d-$permutation $\bsig \in S_n^{d-1}(\treea)$, one has $\gamma^d(\bsig) \in \mathbb{T}^{2^{d-1}}_{n}(F)$. It is however not clear  at the moment that all trees $\cT$ in $\mathbb{T}^{2^{d-1}}_{n}(F)$ have a $d-$permutation $\bpi \in S_n^{d-1}(\treea)$ such that $\gamma^d(\bpi)= \cT$.

\begin{prop}
If a $d-$permutation $\bpi$ avoids $\treea$ and $\treeb$, then it respects ${\bf C}_F$.
\label{prop:impli}
\end{prop}

\begin{proof}

Let $\cT$ be the max-tree of $\bpi$ with respect to the axis $d-1$. By Lemma \ref{lem:dir}, there are no points $p_1$ and $p_2$ such that, $\pi_{d-1}(p_1)< \pi_{d-1}(p_2)$ and ${\bf dir}(p_2,p_1) \notin F$. We thus have by Remark \ref{rem:dir} that there are no internal nodes $r_2$ and $r_1$ such that $r_1 \in \cT_{\bf f}(r_2)$ with ${\bf f} \notin F$.

Let $p_a$, $p_b$ and $p_c$ be three points in $\bpi$ such that their corresponding internal nodes  $r_a$, $r_b$ and $r_c$ in $\cT$ satisfy $r_a \in \cT_{\bf dir^i}(r_c)$ and $r_b \in \cT_{\bf dir^j}(r_c)$ with $i<j$. The implication is then proven by showing that:
\begin{enumerate}
    \item $\pi_l(p_a)<\pi_l(p_c)<\pi_l(p_b)$, \quad if $\bfdir^i_l \neq \bfdir^j_l $.
    \item $\pi_n(p_a)<\pi_n(p_b)<\pi_n(p_c)$, \quad if $\bfdir^i_n = \bfdir^j_n = -$.
    \item $\pi_m(p_c)< \pi_m(p_a)< \pi_m(p_b)$, \quad if $\bfdir^i_m = \bfdir^j_m = +$. 
\end{enumerate}

\medskip

By remark \ref{rem:dir}, if $r_a \in \cT_{\bf dir^i}(r_c)$ and $r_b \in \cT_{\bf dir^j}(r_c)$, one must have that {\bf dir}$(p_c,p_a)= {\bf dir}^i$ and {\bf dir}$(p_c,p_b)= {\bf dir}^j$ in $\bpi$. Let us now prove this item by item.

\textbf{ First item:} By equation \eqref{eq:dirf}, if $\bfdir^i_l \neq \bfdir^j_l $ and $i<j$ one has $\bfdir^i_l=-$ and $\bfdir^j_l=+$. This implies that $\pi_l(p_c)<\pi_l(p_b)$ and $\pi_l(p_a)<\pi_l(p_c)$ and thus that $\pi_l(p_a)<\pi_l(p_c)<\pi_l(p_b)$. 

\medskip
\textbf{ Second item:} If $\bfdir^i_n = \bfdir^j_n = -$, one has either:
\begin{align*}
    &\text{1. } \pi_n(p_b)<\pi_n(p_a)<\pi_n(p_c) \;, \\
    &\text{2. } \pi_n(p_a)<\pi_n(p_b)<\pi_n(p_c) \;.
\end{align*}
Additionally, from the first item and by equation \eqref{eq:dirf}, there exists also an $l<n$ such that $\pi_l(p_a)<\pi_l(p_c)<\pi_l(p_b)$. The configuration $1.$ thus leads to an occurance of $\treeb$.

\textbf{Third item:} If $\bfdir^i_m = \bfdir^j_m = +$, one has either:
\begin{align*}
    &\text{a. } \pi_m(p_c)<\pi_m(p_a)<\pi_m(p_b) \;, \\
    &\text{b. } \pi_m(p_c)<\pi_m(p_b)<\pi_m(p_a) \;.
\end{align*}
Additionally, from the first item and by equation \eqref{eq:dirf}), there exists also an $l>m$ such that $\pi_l(p_a)<\pi_l(p_c)<\pi_l(p_b)$. The configuration $b.$ thus leads to an occurrence of $\treeb$.
\end{proof}

\begin{prop}
If a $d-$permutation $\bpi$ contains an occurrence of $\treea$ it does not respect ${\bf C}$.
\label{prop:impliA}
\end{prop}
\begin{proof}
    Let $\cT$ be the max-tree of $\bpi$ with respect ot the axis $d$. If $\bpi$ contains an occurrence of $\treea$ there are two points $p_a$ and $p_b$ such that
    $$\pi_i(p_a)<\pi_i(p_b)\;,$$$$ \pi_j(p_b)<\pi_j(p_a)\;,$$$$\pi_k(p_a)<\pi_k(p_b)\;,$$
    for some $i<j<k$.
    Let $r_a$ and $r_b$ be the corresponding points of $p_a$ and $p_b$ in $\cT$. There are then three possibles configurations:
    \begin{enumerate}
        \item a) $r_a \in \cT_{\bf f}(r_b)$ or b) $r_b \in \cT_{\bf f}(r_a)$ for some ${\bf f}$ in $\mathbb{F}_k$.
        \item There exists a third internal node $r_c$ in $\cT$ such that $r_a \in \cT_{\bf f_1}(r_c)$ and  $r_b \in \cT_{\bf f_2}(r_c)$. for some directions ${\bf f_1}\neq {\bf f_2}$ in $\mathbb{F}_k$. We denote by $p_c$ the point in $\bpi$ associated with this internal node.
    \end{enumerate}
    The subtree induced by these configurations are illustrated in Figure \ref{fig:illustr_rarb}.

    In the three configurations if {\bf f} or ${\bf f_1}$ or ${\bf f_2}$ is not in $F$ then $\bpi$ doesn't respect ${\bf C}$.

    Due to the occurrence of $\treea$, one has ${\bf dir}(p_a,p_b)= (\ldots, +,\ldots,-,\ldots,+,\ldots)$ (and inversely ${\bf dir}(p_b,p_a)$ is of the form ${\bf dir}(p_b,p_a)=(\ldots, -,\ldots,+,\ldots,-,\ldots)$. This implies that in the first configuration, ${\bf f }$ is not in $F$ and that $\bpi$ thus doesn't respect ${\bf C}$.
    
    In the second configuration, one must have that ${\bf f_1} \in F$ and ${\bf f_2} \in F$. Since ${\bf f_1}\neq {\bf f_2}$ there exist an $n$ such that $$ a) \quad  \pi_n(p_a)<\pi_n(p_c) < \pi_n(p_b) \quad \text{ or } \quad b) \quad \pi_n(p_b)<\pi_n(p_c) < \pi_n(p_a) \;.$$ One has then ${\bf f_1}={\bf dir}^{p}$ and ${\bf f_2}={\bf dir}^{q}$ with $p<q$ in case $(a)$ and $q<p$ in case $(b)$. In case $(a)$ one must then have $\pi_l(p_a)<\pi_l(p_b)$ for all coordinates $l$ if $\bpi$ respects {\bf C}. Similarly in case $(b)$ one must have $\pi_l(p_b)<\pi_l(p_a)$ for all coordinates $l$ if $\bpi$ respects {\bf C}. However, since one has $\pi_j(p_b)<\pi_j(p_a)$ and $\pi_k(p_a)<\pi_k(p_b)$ both cases do not respect ${\bf C}$.
\end{proof}

\begin{prop}
Let $\bpi$ be a $d-$permutation that avoids $\treea$, if it  contains an occurrence of $\treeb$ it does not respect ${\bf C}$.
\label{prop:impliB}
\end{prop}
\begin{proof}
    Let $\cT$ be the max-tree of $\bpi$ with respect to the axis $d$. If $\bpi$ contains an occurrence of $\treeb$ there are then three points $p_a$, $p_b$ and $p_c$ such that

 \begin{equation}
    \label{eq:patt_occ}
    \begin{split}
    \pi_i(p_A)<\pi_i(p_B)< \pi_i(p_C)\;, \\\pi_j(p_C)<\pi_j(p_A)<\pi_j(p_B)\;,
    \end{split}
    \end{equation}
    for some $i<j$. Additionally, one has to consider both the case when $j=d$ and when $j<d$.
    Let $r_A$, $r_B$ and $r_C$ be the internal nodes in $\cT$ corresponding to $p_A$, $p_B$ and $p_C$. Since $\bpi$ avoids $\treea$ there are only internal nodes whose directions are in $F$ in $\cT$. One can distinguish four possible configurations of the nodes $r_A, r_B$ and $r_C$ in $\cT$:
    \begin{enumerate}
        \item $r_A \in \cT_{\bf dir^i}(r_B)$ and $r_C \in \cT_{\bf  dir^j}(r_B)$ for some $i\neq j$, or other permutations of $A,B$ and $C$.
        \item $r_A \in \cT_{\bf dir^i}(r_B)$ and $r_C \in \cT_{\bf  dir^j}(r_A)$ with possibly $i=j$, or other permutations of $A,B$ and $C$.
        \item There exists a node $r_D$ in $\cT$ such that $r_A  \in \cT_{\bf dir^i}(r_D)$, $r_B  \in \cT_{\bf dir^j}(r_D)$ and  $r_C  \in \cT_{\bf dir^k}(r_D)$ with $i \neq j \neq k$.
        \item There exists a node $r_D$ in $\cT$ such that $r_A  \in \cT_{\bf dir^i}(r_D)$, $r_B  \in \cT_{\bf dir^j}(r_D)$ and  $r_C  \in \cT_{\bf dir^k}(r_B)$ with $i \neq j$ and $k$ possibly equal to $i$ or $j$, or other permutations of $r_A$,$r_B$ and $r_C$.
    \end{enumerate}

Each configuration is illustrated in Figures \ref{fig:configs1} and \ref{fig:configs2}.

We prove here that any of these configurations is either not possible or  do not respect ${\bf C}$.

In any configuration, if $r_A\in \cT_{\bf dir^i}(r_C) $ for any ${\bf dir}^i \in F$, one must have that $j< d-1$ and that $\pi_{d-1}(p_A)< \pi_{d-1}(p_C)$ which gives an occurrence of $\treea$ with $p_a$ and $p_C$. Configurations $(1.c)$, $(2.ca)$, $(2.cb)$, $(2.bc)$, and $(4.ca)$ can thus be straightforwardly ruled out as $\bpi$ avoids $\treea$. A similar statement holds if $r_b\in \cT_{\bf dir^i}(r_C)$ for any ${\bf dir}^i \in F$. One can thus also rule out configurations $(2.ac)$ and $(4.cb)$. 

Aditionally, if there exist a point $p_E$ (which can be possibly be the point $p_B$) such that $r_A \in \cT_{\bf dir^x}(r_E)$ and $r_C \in \cT_{\bf  dir^y}(r_E)$ with $x \neq y$, one must have either that $\pi_l(p_A)<\pi_l(p_E)< \pi_l(p_C)$ or $\pi_l(p_C)<\pi_l(p_E)< \pi_l(p_A)$ for some coordinate $l$. In the first case one has $x<y$ and in the second case one has $x>y$. Then, the second line of Equation \ref{eq:patt_occ} implies that the first case do not respect {\bf C} while the first line of Equation \ref{eq:patt_occ} implies that the second case do not respect ${\bf C}$. This implies that configurations $(1.b)$, $(3)$, $(4.ab)$, $(4.ba)$ and $(4.bc)$ can be ruled out. Similarly, if there exist a point $p_E$ (which can be possibly be the point $p_A$) such that $r_B \in \cT_{\bf dir^x}(r_E)$ and $r_C \in \cT_{\bf  dir^y}(r_E)$ with $x \neq y$, one must have either that $\pi_l(p_C)<\pi_l(p_E)< \pi_l(p_B)$ or $\pi_l(p_B)<\pi_l(p_E)< \pi_l(p_C)$ for some coordinate $l$. In the first case one that $y<x$ and in the second case one has that $x<y$. As before, the second and first lines of Equation \ref{eq:patt_occ} implie that the first case  and the second case respectively do not respect ${\bf C}$. This rules out configurations $(1.a)$ and $(4.ac)$.


Finally, because of the first line of equation \ref{eq:patt_occ}, one has {\bf dir}$(p_B$,$p_A) \neq {\bf dir}(p_B$,$p_C)$ which implies that one cannot have  $r_A \in \cT_{\bf dir^i}(r_B)$ and $r_C \in \cT_{\bf  dir^i}(r_B)$. This rules out configurations $(2.ba)$. Similarly, because of the second line of equation \ref{eq:patt_occ}, one has {\bf dir}$(p_A$,$p_B) \neq {\bf dir}(p_A$,$p_C)$ which implies that one cannot have  $r_C \in \cT_{\bf dir^i}(r_A)$ and $r_B \in \cT_{\bf  dir^i}(r_A)$. This rules out configurations $(2.ab)$.

We thus have that if the points $p_a$, $p_b$ and $p_c$ are an occurrence of $\treeb$ in $\bpi$, all possible configurations of the points $r_a$, $r_b$ and $r_c$ in $\cT$ either do not respect ${\bf C}$ or are not possible configurations, which conclude the proof. 

\end{proof}

By Propositions \ref{prop:impli}, \ref{prop:impliA} and \ref{prop:impliB} we obtain the following Lemma which conlude the proof.

\begin{lem}
    The set $\treeset$ corresponds to the set of $d-$permutations that respect ${\bf C}$.
    \label{lem:respect_eq}
\end{lem}

\section{Conclusion}
\label{sec:Conclusion}

In this paper, we generalized the notion of max-tree of a permutation to $d-$permutations. This construction induces a mapping from $d-$permutations to $2^{d-1}-$ary trees. We then showed that, given a compatible set of total orders on $\mathbb{F}^d$, one can define an \emph{admissible} permutation class for which this mapping is bijective. This yields a bijection between these admissible classes and $2^{d-1}-$ary trees. 

By restricting this general bijection, we answered a conjecture formulated in \cite{bonichon2022baxter} and \cite{muller2025study}. More precisely, we proved that the $d-$permutation avoiding $\treea$ and $\treeb$ are in bijection with $d-$ary trees which also solves the enumeration problem for this class of $d-$permutations.

As follow-up of this work, several open problems appear interesting to us. 

\paragraph{Question 1:}
The bijection between $\treeset$ and the set of $d-$ary trees is obtained by constructing the max-trees with respect to the last axis of the permutation in $\treeset$. As mentioned in Section \ref{sec:section1}, a similar construction can be defined using any other axis of a $d-$permutation. For each such axis, one can similarly define admissible classes of $d-$permutations by considering total orderings of the corresponding negative directions.

In this context, a natural open question is whether one can characterize, in terms of forbidden patterns, other classes of $d-$permutations that are in bijection with $d-$ary trees through a construction based on an axis different from the last axis ? 

\paragraph{Question 2:}
The bijection between $\treeset$ and the set of $d-$ary trees arises as a restriction of the general mapping $\gamma^d$ to $d-$permutations whose max-trees avoid certain directions. This restriction allows the arity of the trees to match the dimension of the permutations. A second open question is whether there exists a more general class of $d-$permutations, in bijection with $2^{d-1}-$ary trees that can be characterized in terms of forbidden patterns.

\paragraph{Question 3:}
In the construction of the max-trees of a $d-$permutation, one considers the maximal point of the permutation with respect to the last axis. One may instead consider the dual construction based on the minimal point of a $d-$permutation, leading to the notion of min-tree. In the classical case of permutations, the pair (min-tree, max-tree) is known to be a pair of twin binary trees. Moreover, such pairs of trees are known to be in bijection with Baxter permutations.

This raises the following question: can analogous phenomena be identified in the higher-dimensional setting ? In particular, can we identify structural conditions on the pair (min-tree, max-tree) of a $d-$permutation and can we identify permutations classes that are in bijection with these pairs of trees ? 

\paragraph{Question 4:}
Using the inverse mapping described in appendix \ref{sec:tree2perm}, one can generate uniformly large random $d-$permutations (of a given size) in $\treeset$ and more generally $d-$permutations that are admissible with respect to a compatible set of total orders. Examples of such large $3-$permutations are shown in Figure \ref{fig:random_231} for $\treeset$ and in Figure \ref{fig:random_order} for a $3-$permutation admissible with the compatible set of total order ${\bf C_{ex}}=C_{ex,0}, C_{ex,1}, C_{ex,2}$ where 
$$ C_{ex,0}: (-+-)< (---) < (++-) < (+--) , \quad  C_{ex,2}: (---)< (+--) < (-+-) < (++-) ,$$
$$ C_{ex,2}: (---)< (++-) < (+--) < (-+-) .$$

It would thus be interesting to study the scaling limits of such classes of $d-$permutations as it was done in \cite{borga2025} for the separable $d-$permutations.

\begin{figure}[!htb]
        \center{ \minipdf{0.40}{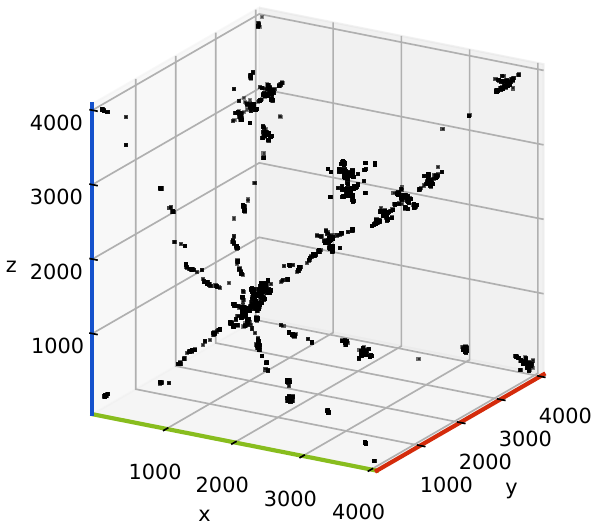} \minipdf{0.40}{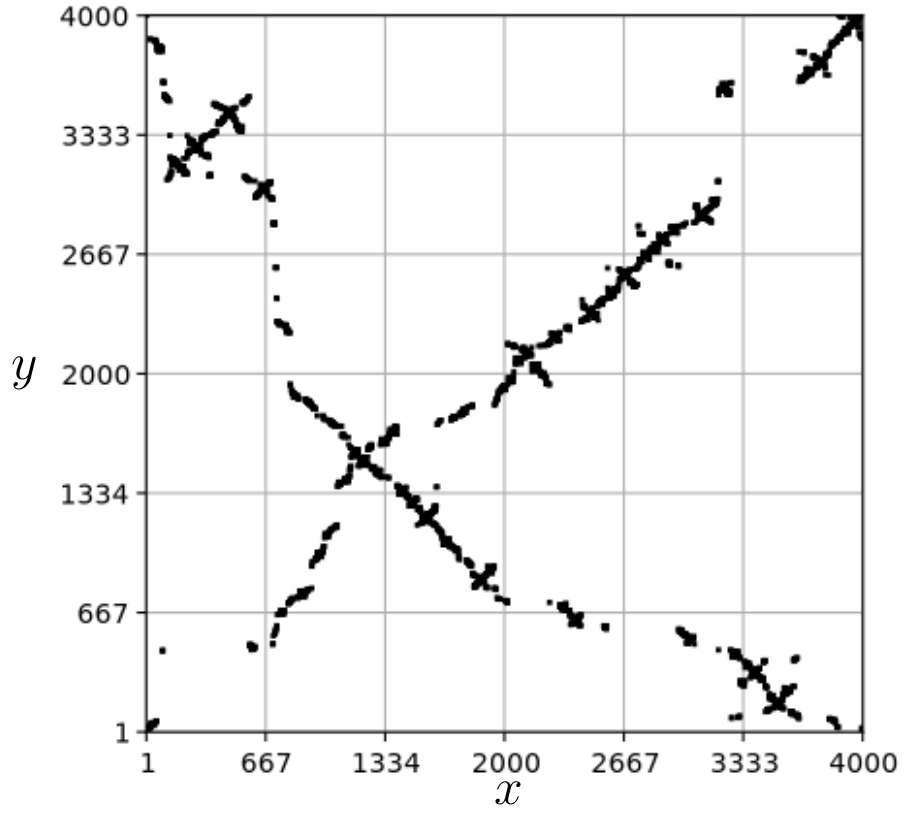} \minipdf{0.40}{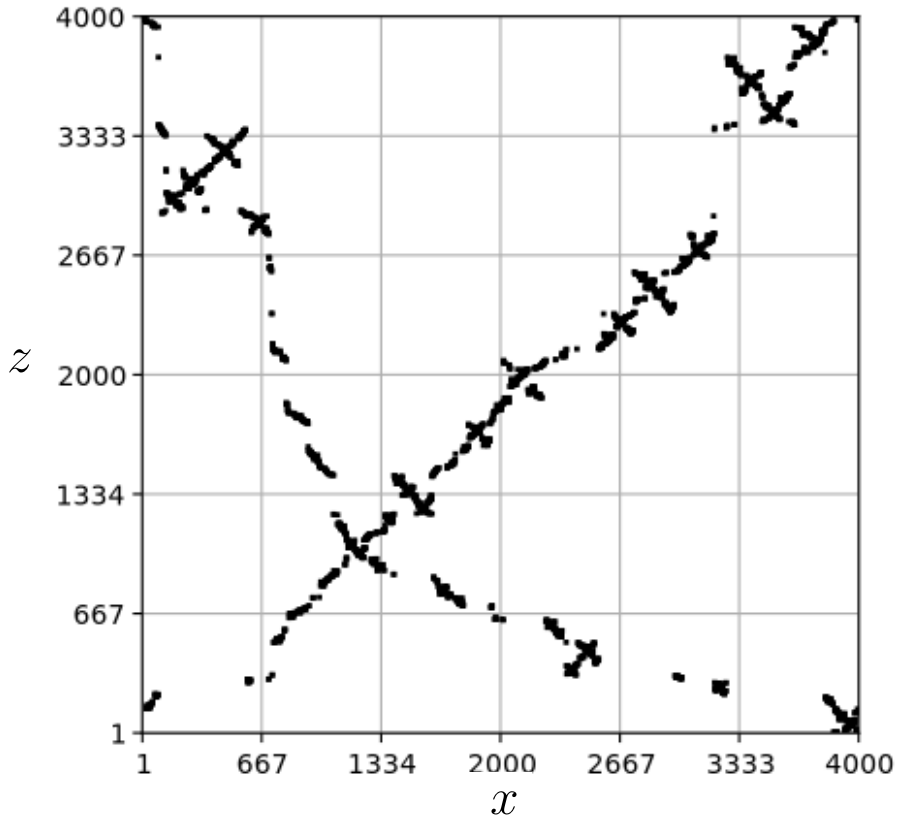} \minipdf{0.40}{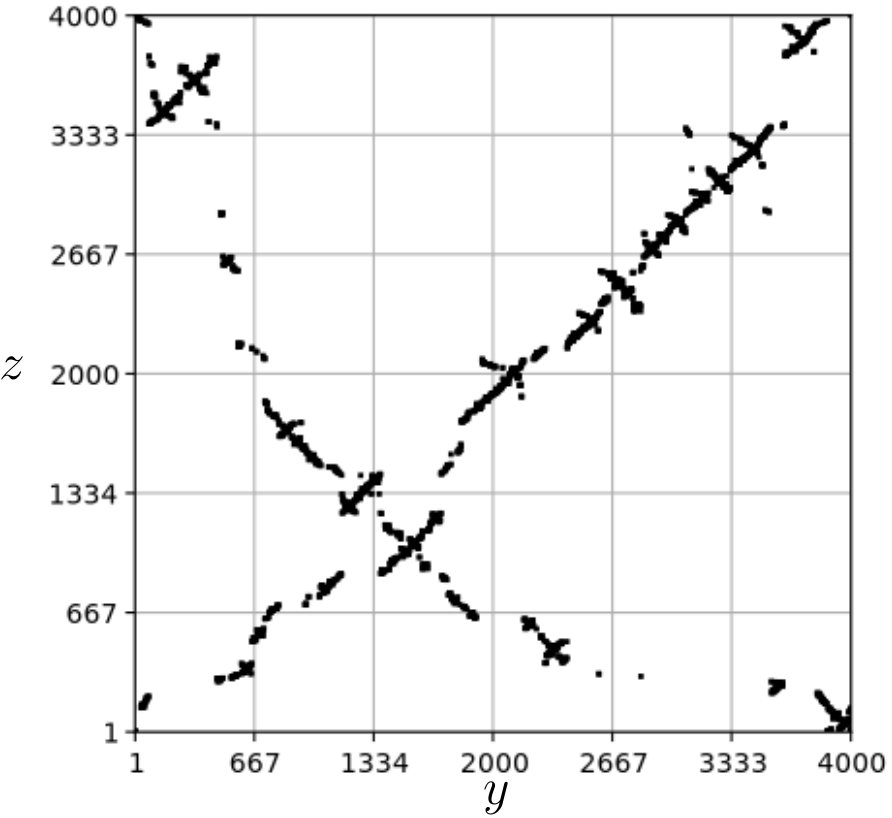}}
        \caption{A random $3-$permutations of size 4000 in $S^{2}_{n,{\bf C_{ex}}}$.}
        \label{fig:random_order}
    \end{figure}

\bibliographystyle{alpha}
\bibliography{sample}


\appendix

\section{Node configurations of Propositions \ref{prop:impliA} and \ref{prop:impliB}}

In this appendix, we illustrate and label the various possible configurations of internal nodes that appear in the proofs of Propositions \ref{prop:impliA} and \ref{prop:impliB}. 
Figure \ref{fig:illustr_rarb} shows the configurations involved in the proof of Proposition \ref{prop:impliA} and Figures \ref{fig:configs1} and \ref{fig:configs2} the ones of Proposition \ref{prop:impliB}. 

For each configuration, we depict only the smallest subtree containing the internal nodes involved. Consequently, the root of the displayed subtree does not necessarily coincide with the root of the entire tree. Moreover, when an internal node is represented by a triangular shape, this indicates that the relationship between this node and the node above it is an ancestor–descendant relation, rather than necessarily a direct parent–child relation. In this case, the label on the connecting arrow specifies the direction of the subtree of the ancestor in which the descendant lies.

\begin{figure}[!htb]
        \center{1.a)\;  \minipdf{0.25}{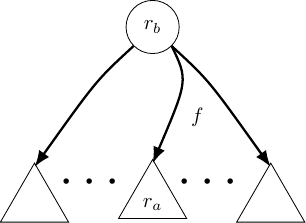} 1.b) \; \minipdf{0.25}{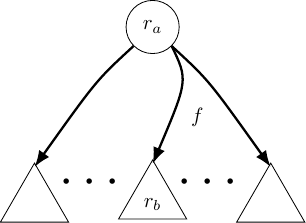}  2) \; \minipdf{0.28}{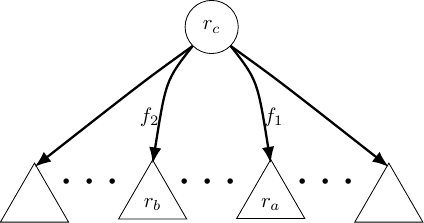}}
        \caption{Possible configurations of two internal nodes.}
        \label{fig:illustr_rarb}
    \end{figure}

\begin{figure}[!htb]
        \center{1.a)\minipdf{0.25}{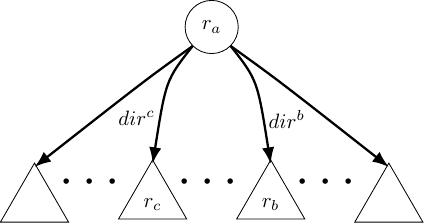} 1.b)\minipdf{0.25}{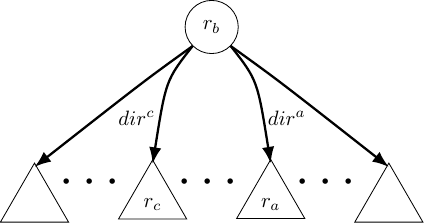}  1.c)\minipdf{0.25}{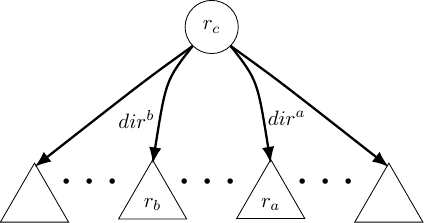} \newline 2.ab)\minipdf{0.23}{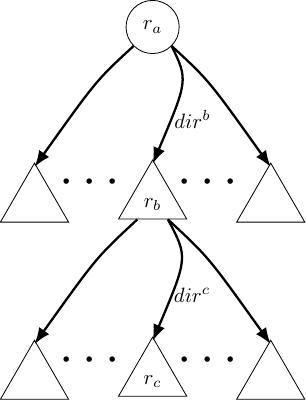}   2.ac)\minipdf{0.23}{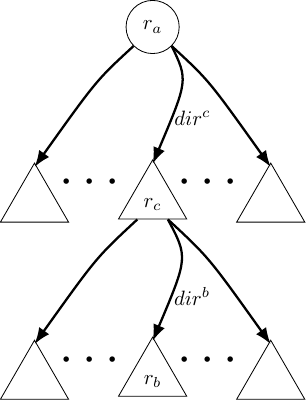} 2.ba) \minipdf{0.23}{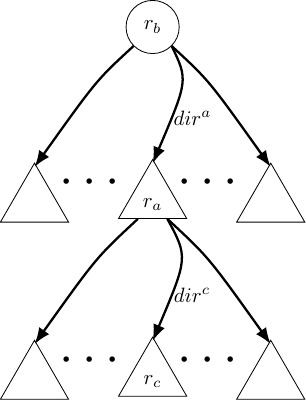} \newline2.bc)\minipdf{0.23}{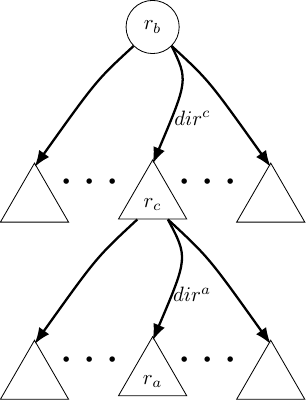}  2.ca)\minipdf{0.23}{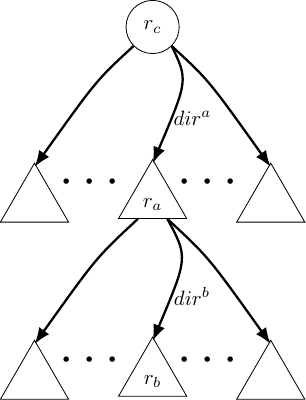}  2.cb)\minipdf{0.23}{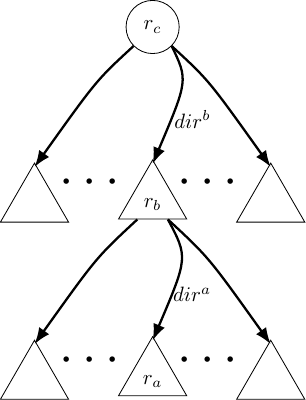}}
        \caption{Possible configurations of three internal nodes (1st part).}
        \label{fig:configs1}
    \end{figure}

\begin{figure}[!htb]
        \center{3.a)\minipdf{0.25}{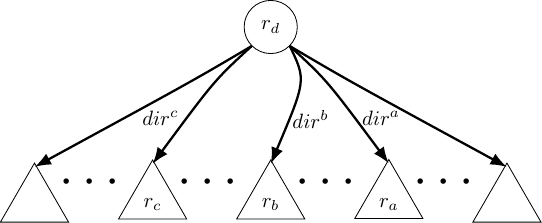}    4.ab)\minipdf{0.22}{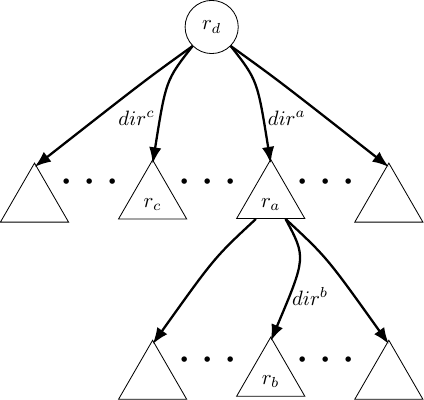}   4.ac)\minipdf{0.22}{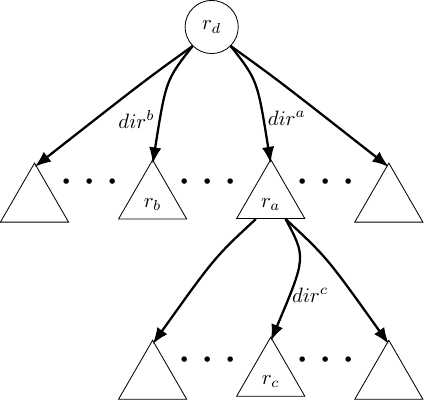} \newline 4.ba) \minipdf{0.25}{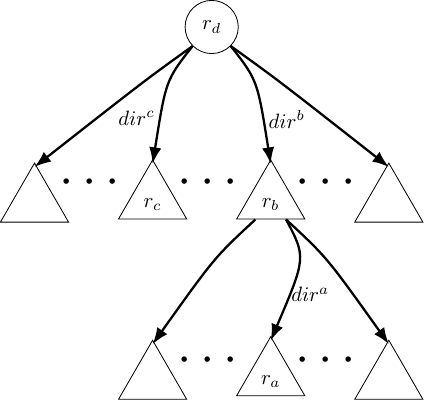} 4.bc)\minipdf{0.22}{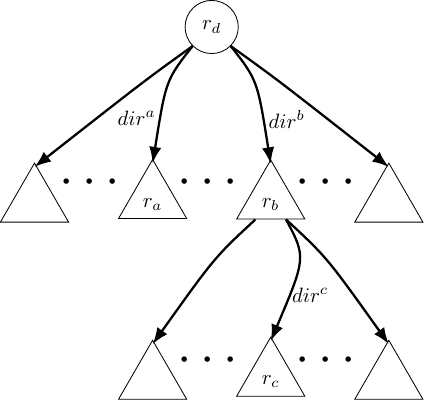} 4.ca)\minipdf{0.22}{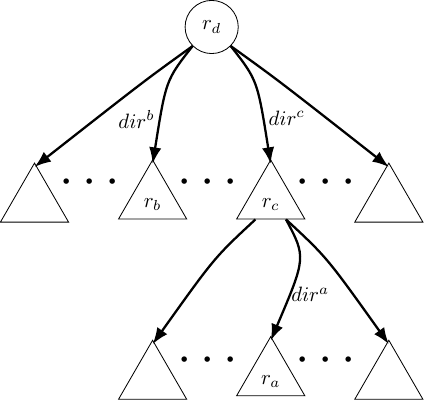} 4.cb)\minipdf{0.22}{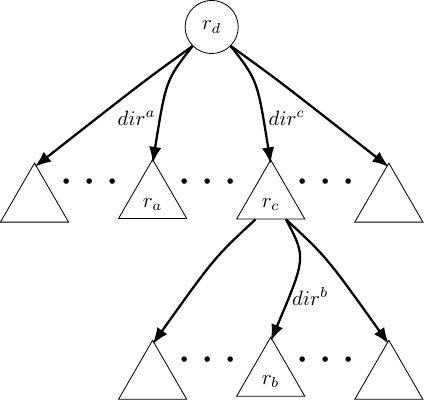}}
        \caption{Possible configurations of three internal nodes (2nd part).}
        \label{fig:configs2}
    \end{figure}

\FloatBarrier

\section{From trees to $d-$permutations}
\label{sec:tree2perm}

\begin{definition}
\label{def:ordertree}
Let $r_A$ and $r_B$ be two internal nodes in a
$2^{d-1}$-ary tree $\cT \in \mathbb{T}^{2^{d-1}}_{n}$ and let $C$
be a total order of the elements of $\mathbb{F}^d$ that is compatible with
respect to the axis $i$.

We say that $r_A$ precedes $r_B$ with respect to $C$ and the axis $i$,
denoted $r_A <_{C,i} r_B$, if:

\begin{itemize}
    \item $r_A=r_B$;
    \item $r_A\in \cT_f(r_B)$ and $f_i=-$;
    \item $r_B\in \cT_f(r_A)$ and $f_i=+$;
    \item $r_A$ is not an ancestor of $r_B$ (nor the converse), and for
    their first common ancestor $r_C$ one has $r_A\in \cT_{f_1}(r_C)$, $r_B\in \cT_{f_2}(r_C)$, where $f_1$ precedes $f_2$ in $C_i$.
\end{itemize}
\end{definition}

The relation $<_{C,i}$ is a total order of the internal nodes of a tree. It can thus be used to define a permutation.
\begin{definition}
\label{def:tree2perm}
Let $\cT$ be a $2^{d-1}$-ary tree in
$\mathbb{T}^{2^{d-1}}_{n}$ and let
${\bf C}=(C_0,\ldots,C_{d-1})$ be a compatible set of total orders
on $\mathbb{F}^d$. For $i=0,\ldots,d-1$, let $\sigma_i$ be the list of internal nodes
ordered according to the relation $<_{C_i,i}$. The \emph{ordered $d$-permutation of $\cT$ with respect to ${\bf C}$} is the
$d$-permutation
\[
\pi_{\bf C}(T)
=
(\pi_{1,{\bf C}},\ldots,\pi_{d-1,{\bf C}})
\]
defined such that $\pi_{j,{\bf C}}(i)$ is the position in $\sigma_{j}$ of
the node whose position is $i$ in $\sigma_0$.
\end{definition}


In the ordered $d$-permutations associated
with a tree $\cT$, each point corresponds to an internal node of $\cT$.

Given a tree
$\cT\in \mathbb{T}^{2^{d-1}}_{n}$
and a compatible total order $C$
with respect to the axis $l$
on the elements of $\mathbb{F}^d$,
one can obtain the list of internal nodes of $\cT$
ordered according to the relation $<_{C,l}$
by using Algorithm \ref{alg:Tree2perm}. Before explaining in more detail this algorithm,
let us use a few notations:

\begin{itemize}
\item Given a tree $\cT$ we denote by $\mathrm{root}(\cT)$ the root of this tree.
\vspace{-0.5 em}
\item Given an internal node $r$ in a tree,
we denote by $\mathrm{par}(r)$ its parent and by
$\mathrm{child}_{\bf dir}(r)$
its child with associated direction ${\bf dir}$.
\vspace{-0.5 em}
\item Given a total order $C$ on the elements of $\mathbb{F}^d$,
we denote by $C(j)$ the direction that has position $j$
in this total order.
\vspace{-0.5 em}
\item An internal node is said to be maximal if all its children are leaves.

\end{itemize}

Algorithm \ref{alg:Tree2perm} performs a peeling of $\cT$
where the internal nodes whose directions have the lowest
position in $C$ are removed first. This algorithm is composed of three main phases:
descending, removing nodes, and backtracking. These three phases are repeated until all internal nodes of the tree
have been deleted.

\begin{itemize}

\item \textbf{Descending phase:}
Starting from the root of the current subtree,
the algorithm recursively moves to the child whose direction
has the smallest position in the total order $C$,
as long as this child is an internal node.
During this descent, non-maximal internal nodes are added
to the output list $\pi_l$ when the algorithm enters for the
first time one of their subtrees whose associated direction
has a positive $l$-th coordinate.
The descent stops when a maximal internal node is reached.
This phase is given by the subroutine described in Algorithm \ref{alg:Descent}.

\item \textbf{Node removal phase:}
Once a maximal internal node is reached,
it is removed by replacing its whole subtree with a single leaf.
During this step, the corresponding internal node is added to
$\pi_l$ if it was not already present.

\item \textbf{Backtracking phase:}
After a removal, the algorithm backtracks to the parent
of the last deleted node.
If this parent has now become maximal in the remaining tree,
it is removed and added to $\pi_l$.
Otherwise, the algorithm resumes the descending phase.

\end{itemize}

\begin{prop}
Let $\cT$ be a $2^{d-1}$ ary tree in $\mathbb{T}^{2^{d-1}}_{n}$, and let $C$ be a total order of the elements of $\mathbb{F}^d$ that is compatible with respect to the axis $l$. The output of algorithm \ref{alg:Tree2perm} with input parameter $\cT$, $C$ and axis $l$ is the list of internal nodes of $\cT$ ordered according to $<_{C,l}$.
\end{prop}

\begin{proof}
Let $r_a$ and $r_b$ be two internal nodes in $\cT$ and let
$\pi_l(r_a)$ and $\pi_l(r_b)$ be the position of $r_a$
and $r_b$ in $\pi_l$.
We prove that in each configuration of $r_a$ and $r_b$
such that $r_a <_{C,l} r_b$, one has
$\pi_l(r_a)<\pi_l(r_b)$.

If $r_a\in T_{\bf f}(r_b)$ and ${\bf f}_i=-$, during the course of
Algorithm \ref{alg:Tree2perm}, $r_b$ is not added to $\pi_l$ when the
algorithm enters its subtree containing $r_a$.
Thus $r_a$ is removed from $\cT$ before $r_b$ is added to
$\pi_l$ and one has
$\pi_l(r_a)<\pi_l(r_b)$.

If $r_b\in T_{\bf f}(r_a)$ and ${\bf f}_i=+$ one that $r_a$
is added to $\pi_l$ when Algorithm \ref{alg:Tree2perm} enters its subtree
containing $r_b$.
One thus also has
$\pi_l(r_a)<\pi_l(r_b)$.

If $r_a$ is not an ancestor of $r_b$ (or the inverse)
and for their first common ancestor $r_c$ one has
$r_a\in \cT_{\bf f_1}(r_c)$ and
$r_b\in \cT_{\bf f_2}(r_c)$
with ${\bf f_1}$ that precedes ${\bf f_2}$ in $C_i$,
Algorithm \ref{alg:Tree2perm} enters first the subtree of $r_c$
that contains $r_a$.
Thus $r_a$ is removed before Algorithm \ref{alg:Tree2perm} enters
the subtree of $r_c$ that contains $r_b$ and one has
$\pi_l(r_a)<\pi_l(r_b)$.
\end{proof}

\begin{rem}
\label{rem:ancestors}
Each point of a $d$-permutation $\bpi$ corresponds to an internal
node in its max-tree $\cT_{\max}(\bpi)$. Let $p_a$ and $p_b$ be two points of $\bpi$, and let
$r_a$ and $r_b$ be their corresponding internal nodes in
$\cT_{\max}(\bpi)$.
Let also $r_{c_1},\ldots,r_{c_m}$
be the ancestors of $r_b$ in $\cT_{\max}(\bpi)$, and let $p_{c_1},\ldots,p_{c_m}$ be their corresponding points in $\bpi$.

Then $r_a \in \cT_{\bf f}(r_b)$, ${\bf f} \in \mathbb{F}^d,$
if and only if ${\bf dir}(p_b,p_a)={\bf f}$,
and, for every $i=1,\ldots,m$, ${\bf dir}(p_{c_i},p_a)
= {\bf dir}(p_{c_i},p_b)$.
\end{rem}

\begin{algorithm}
\caption{Descent \\ {\bf Input:} A $2^{d-1}$-ary tree $\cT\in \mathbb{T}^{2^{d-1}}_{n}$,
an axis $l$,
a compatible total order $C$ with respect to $l$ on the elements of $\mathbb{F}^d$,
an integer $m$ and an ordered list of internal nodes $\pi_l$. \\
{\bf Output:} Two integers $m'$ and $a$, an internal node $act$,
an ordered list of internal nodes $\pi'_l$.}
\label{alg:Descent}
\begin{algorithmic}[1]

\State Set the current node $act$ to the root of $T$
\State $j \gets 0$ and $a \gets 0$

\While{$act$ is not a maximal internal node of $\cT$}

    \If{$act \neq \mathrm{root}(\cT)$ and ${\bf dir}_l(act)=+$ and
        $\mathrm{par}(act)\notin \pi_l$}
        \State Add $\mathrm{par}(act)$ to $\pi_l$ at position $m$
        \State $m \gets m+1$
    \EndIf

    \If{$\mathrm{child}_{C(j)}(act)$ is not a leaf}
        \State $act \gets \mathrm{child}_{C(j)}(act)$
        \State $a \gets j$
        \State $j \gets 0$
    \Else
        \State $j \gets j+1$
    \EndIf

\EndWhile
 \If{$act \neq \mathrm{root}(\cT)$ and ${\bf dir}_l(act)=+$ and
        $\mathrm{par}(act)\notin \pi_l$}
        \State Add $\mathrm{par}(act)$ to $\pi_l$ at position $m$
        \State $m \gets m+1$
    \EndIf
    
\State \Return $m,a,act,\pi_l$
\end{algorithmic}
\end{algorithm}

\begin{algorithm}
\caption{Tree2perm \\ {\bf Input: } A $2^{d-1}$-ary tree $\cT$ with $n$ internal nodes in
$\mathbb{T}^{2^{d-1}}_{n}$, an axis $l$ and a total order $C$ of the directions in $\mathbb{F}^d$ that is compatible with respect to $l$   \\ {\bf Output: } A permutation $\pi_l$}
\label{alg:Tree2perm}
\begin{algorithmic}[1]

\State $m\gets 1$, $\pi_l\gets \emptyset$, $a\gets 0$
\State Set $act$ to the root of $\cT$

\While{$m<n+1$}

    \If{$act$ is an internal node}

        \If{$act\neq \mathrm{root}(\cT)$ and $\mathrm{dir}_l(act)=+$ and
            $\mathrm{par}(act)\notin \pi_l$}
            \State Add $\mathrm{par}(act)$ to $\pi_l$ at position $m$
            \State $m\gets m+1$
        \EndIf

        \If{$act$ is not maximal}
            \State $(m,a,act,\pi_l)\gets
            \mathrm{Descent}(\cT(act),l,C,m,\pi_l)$
        \EndIf

        \If{$act\notin \pi_l$}
            \State Add $act$ to $\pi_l$ at position $m$
            \State $m\gets m+1$
        \EndIf

        \State In $\cT$, replace the subtree $\cT(act)$ by a single leaf

    \EndIf

    \State $a\gets a+1$

    \If{$a>2^{d-1}-1$}
        \State $act\gets \mathrm{par}(act)$
        \State $a\gets$ index $k$ such that $C_l(k)=\mathrm{dir}(act)$
    \Else
        \State $act\gets \mathrm{child}_{C(a)}(\mathrm{par}(act))$
    \EndIf

\EndWhile

\State \Return $\pi_l$
\end{algorithmic}
\end{algorithm}

\begin{rem}
\label{rem:treedir}
Let $\cT$ be a $2^{d-1}$-ary tree in $T_{n}^{2^{d-1}}$, and let ${\bf C}=(C_0,\ldots,C_{d-1})$
be a compatible set of total orders on $\mathbb{F}^d$.
Let $r_1$ and $r_2$ be two internal nodes of $\cT$, and let
$p_1$ and $p_2$ be their corresponding points in
$\pi_{\bf C}(\cT)$. If there exists a direction ${\bf f}\in \mathbb{F}^d$ such that $r_1\in T_{\bf f}(r_2)$, then ${\bf dir}(p_2,p_1)={\bf f}$.
\end{rem}

\begin{lem}
Let $\cT$ be a $2^{d-1}$-ary tree in $\mathbb{T}^{2^{d-1}}_{n}$, and let
${\bf C}=(C_0,\ldots,C_{d-1})$ be a compatible set of total orders on
$\mathbb{F}^d$. The ordered $d$-permutation of $\cT$ with respect to $C$
satisfies $\gamma^d(\pi_C(\cT)) = \cT$.
\end{lem}

\begin{proof}
By Definition~\ref{def:tree2perm} and Remark~\ref{rem:ancestors}, each internal node of $\cT$
corresponds to a point of $\pi_C(\cT)$ and to an internal node of
$\gamma^d(\pi_C(\cT))$.

Let $r_1$ and $r_2$ be two internal nodes of $\cT$, and let
$r_1'$ and $r_2'$ denote their corresponding internal nodes in
$\gamma^d(\pi_C(T))$. We prove that $\gamma^d(\pi_C(T)) = \cT$ by induction on the depth of $r_2$. More precisely, we show that
for every pair of internal nodes $r_1,r_2$ in $\cT$, if $r_1$ is the
child of direction ${\bf f}$ of $r_2$, then $r_1'$ is also the child of
direction ${\bf f}$ of $r_2'$ in $\gamma^d(\pi_C(\cT))$.

\medskip

If $r_2$ is the root of $\cT$ (depth $0$).
The corresponding point $p_2$ is then the maximal point of
$\pi_C(\cT)$ with respect to axis $k$. Moreover, for every internal node $r_n$ of $\cT$, there exists a
direction ${\bf f_n}\in \mathbb{F}^d$ such that $r_n \in T_{\bf f_n}(r_2)$.
By Remark~\ref{rem:treedir}, the points $p'$ satisfying $\operatorname{dir}(p_2,p')={\bf f}$ are exactly those associated with internal nodes lying in the
subtree of direction ${\bf f}$ of $r_2$. Since $r_1$ is the child of direction ${\bf f}$ of $r_2$, every internal
node $r_3\neq r_1$ in the subtree of direction ${\bf f}$ of $r_2$
belongs to a subtree of direction ${\bf f_3}$ of $r_1$.
Because ${\bf f_3} \in \mathbb{F}^d$, Remark~\ref{rem:treedir} implies that
$\pi_{C,k}(p_3) < \pi_{C,k}(p_1)$. Furthermore, $\operatorname{dir}(p_2,p_1)={\bf f}$, hence $p_1$ is the maximal point of the subpermutation
$\pi_C^{\bf f}$ with respect to axis $k$.
Therefore, in $\gamma^d(\pi_C(T))$, the node $r_1'$
is the child of direction $f$ of $r_2'$.

\medskip

Assume now that $r_2$ is an internal node of depth $w>0$. Let $r_{c,1},\ldots,r_{c,w}$
be the ancestors of $r_2$ in $\cT$. We suppose by induction that $r_{c,1}',\ldots,r_{c,w}'$ the associated nodes of $r_{c,1}',\ldots,r_{c,w}'$ are the ancestors of $r_2'$ in $\gamma^d(\pi_C(T))$.


We prove the following statements:

\begin{enumerate}
\item In $\gamma^d(\bpi_C(T))$, the node $r_1'$ lies in the
subtree of direction ${\bf f}$ of $r_2'$.
\vspace{-0.5 em}
\item The internal nodes lying in the subtree of direction ${\bf f}$
of $r_2'$ correspond exactly to internal nodes of $\cT$
lying in the subtree of direction $f$ of $r_2$.
\vspace{-0.5 em}
\item For every internal node $r$ in the subtree of direction ${\bf f}$
of $r_2$, the associated point $p$ satisfies $\pi_{C,k}(p) < \pi_{C,k}(p_1)$.
\end{enumerate}

Together, (1) and (2) show that the points corresponding to nodes of the subtree of direction ${\bf f}$ of $r_2'$ are exactly those associated with nodes of the subtree of direction ${\bf f}$ of $r_2$. Statement (3) then implies that among all these points, $p_1$
has the greatest last coordinate.
Consequently, $r_1'$ must be the child of direction $f$
of $r_2'$.

\medskip

\textbf{Proof of (1).}
By Remark~\ref{rem:treedir}, $\operatorname{dir}(p_2,p_1)={\bf f}$.
Since $r_1\in \cT_{\bf f}(r_2)$, for every ancestor $r_{c,i}$ of $r_2$
such that $r_2\in \cT_{f_{c,i}}(r_{c,i})$,
one also has $r_1\in \cT_{f_{c,i}}(r_{c,i})$. Let $p_{c,i}$ be the point associated with $r_{c,i}$.
Then $\operatorname{dir}(p_{c,i},p_1)
=
\operatorname{dir}(p_{c,i},p_2)$. By the induction hypothesis,
$r_{c,i}'$ is an ancestor of $r_2'$ in
$\gamma^d(\bpi_{\bf C}(T))$.
Remark~\ref{rem:ancestors} therefore implies that $r_1' \in \gamma^d(\bpi_{\bf C}(T))_{\bf f}(r_2')$.

\medskip

\textbf{Proof of (2).}
In $\cT$, let $r_3$ be an internal node that is not a descendant of $r_2$,
and let $r_q$ be the first common ancestor of $r_2$ and $r_3$. Then $r_3\in T_{\bf f_3}(r_q)$, $r_2\in T_{\bf f_2}(r_q)$ and $f_2\neq f_3$.

Let $p_3$ and $p_q$ be the corresponding points in $\bpi_{\bf C}(\cT)$. By Remark~\ref{rem:treedir}, one has $\operatorname{dir}(p_q,p_3)
\neq
\operatorname{dir}(p_q,p_2)$. By induction hypothesis, the point $p_q$ corresponds to an ancestor of $r_2'$ in $\gamma^d(\bpi_{\bf C}(T))$, this implies by Remark \ref{rem:ancestors} that the internal node associated to  $r_3$ in $\gamma^d(\bpi_{\bf C}(T))$ is not a descendant of $r_2'$ in $\gamma^d(\bpi_{\bf C}(T))$.

Moreover, for any internal node $r_4$ such that $r_4 \in \cT_{\bf f_4}(r_2)$ in $\cT$, one has by Remark \ref{rem:treedir}, that ${\bf dir}(p_2,p_4) = {\bf f_4} \neq {\bf f_2}$ where $p_4$ is the point in $\bpi_{\bf C}(T)$ associated to $r_4$. By Remark \ref{rem:ancestors} one thus has that the internal nodes associated to $r_4$ in $\gamma^d(\bpi_{\bf C}(T))$, do not lie in the subtree of direction {\bf f } of $r_2'$.

\textbf{Proof of (3).}
In $\cT$, since $r_1$ is the child of direction ${\bf f}$ of $r_2$,
every internal node $r$ in the subtree of direction ${\bf f}$
of $r_2$, different from $r_1$, belongs to a subtree of
direction ${\bf f_r}$ of $r_1$. Let $p_r$ be the corresponding point in $\bpi_{\bf C}(T)$ of such internal node.  Since $f_r\in \mathbb{F}^d$, Remark~\ref{rem:treedir} yields $\pi_{C,k}(p_r) < \pi_{C,k}(p_1)$.
\end{proof}

Finally, it can be seen by Definitions \ref{def:ordertree}, \ref{def:tree2perm} and \ref{def:perm2tree} that for any tree $\cT\in \mathbb{T}_{n}^{2^{d-1}}$ and for any compatible set of total orders {\bf C}, the ordered $d-$permutation of $\cT$ with respect to {\bf C} is admissible with respect to ${\bf C}$.

\end{document}